\documentclass[11pt,leqno]{amsart}

\usepackage[T1]{fontenc}
\usepackage[utf8]{inputenc}
\usepackage[british]{babel} 

\usepackage[a4paper,margin=30mm,includeheadfoot]{geometry}
\usepackage{microtype}
\usepackage{fancyhdr}
\usepackage{mathtools}  
\usepackage{amssymb,amsfonts,amsthm}
\usepackage{mathrsfs}
\usepackage{bm}
\numberwithin{equation}{section}
\allowdisplaybreaks

\usepackage{graphicx}   
\usepackage{tikz}
\usepackage{tikz-cd}
\usetikzlibrary{calc,positioning}
\tikzcdset{arrow style=tikz}
\usepackage{tikz-3dplot}

\usepackage{enumitem}
\setlist[enumerate]{label=\textnormal{(\arabic*)},leftmargin=*}

\usepackage{xcolor}
\definecolor{cadmiumgreen}{rgb}{0.0,0.42,0.24}
\definecolor{darkred}{rgb}{0.65, 0.0, 0.0}
\definecolor{darkblue}{rgb}{0,0,0.65}

\usepackage[colorlinks,linkcolor=cadmiumgreen,citecolor=cadmiumgreen,urlcolor=cadmiumgreen]{hyperref}
\usepackage{bookmark} 
\hypersetup{
  pdfauthor={Seungkyu Lee},
  pdftitle={SL-TMSV},
  pdfstartview=FitV
}

\usepackage{aliascnt}

\theoremstyle{plain}
\newtheorem{theorem}{Theorem}[section]

\newaliascnt{proposition}{theorem}
\newtheorem{proposition}[proposition]{Proposition}
\aliascntresetthe{proposition}

\newaliascnt{lemma}{theorem}
\newtheorem{lemma}[lemma]{Lemma}
\aliascntresetthe{lemma}

\newaliascnt{corollary}{theorem}
\newtheorem{corollary}[corollary]{Corollary}
\aliascntresetthe{corollary}

\newaliascnt{claim}{theorem}

\aliascntresetthe{claim}

\theoremstyle{definition}
\newaliascnt{definition}{theorem}
\newtheorem{definition}[definition]{Definition}
\aliascntresetthe{definition}

\theoremstyle{remark}
\newaliascnt{remark}{theorem}
\newtheorem{remark}[remark]{Remark}
\aliascntresetthe{remark}

\newaliascnt{example}{theorem}

\aliascntresetthe{example}

\usepackage[nameinlink,capitalize]{cleveref}
\crefname{theorem}{Theorem}{Theorems}
\crefname{proposition}{Proposition}{Propositions}
\crefname{lemma}{Lemma}{Lemmas}
\crefname{corollary}{Corollary}{Corollaries}
\crefname{claim}{Claim}{Claims}
\crefname{definition}{Definition}{Definitions}
\crefname{remark}{Remark}{Remarks}
\crefname{example}{Example}{Examples}
\crefname{equation}{Equation}{Equations}
\crefname{section}{Section}{Sections}

\DeclareMathOperator{\Hom}{Hom}

\newcommand{\precd}{\prec \!\!\! \cdot \;}

\newcommand{\msc}[1]{\href{https://mathscinet.ams.org/msc/msc2020.html?t=&s=#1}{#1}}

\usepackage[color=yellow!30,textsize=small]{todonotes}

\title[Tropical matroid Schubert varieties]{Tropical matroid Schubert varieties and the graded M\"obius algebra}
\author{Seungkyu Lee}
\address{LMO, Universit\'e Paris-Saclay \& CMLS, \'Ecole Polytechnique}
\email{\href{mailto:seungkyu.lee@universite-paris-saclay.fr}{seungkyu.lee@universite-paris-saclay.fr}}
\date{}

\begin{document}

\begin{abstract}
  We introduce tropical matroid Schubert varieties, a tropical analogue of arrangement Schubert varieties associated with realisable matroids. We prove that the tropical cohomology ring of the tropical matroid Schubert variety associated to any matroid $M$ is isomorphic to the graded M\"obius algebra 
  $\operatorname{B}^\bullet(M)$.  This yields a geometric model for $\operatorname{B}^\bullet(M)$, extending the geometric setting of arrangement Schubert varieties to arbitrary matroids,  including non-realisable ones.
\end{abstract}
\maketitle
\thispagestyle{fancy}
\tableofcontents

\section{Introduction}\label{intro}
The top-heavy conjecture posed by Dowling and Wilson \cite{Dowling1974,Dowling1975} was proved for realisable matroids 
by Huh and Wang in \cite{Huh2017}. 
The key idea of their approach is that the Hilbert series of the cohomology ring of an arrangement Schubert variety $Y_{\mathcal A}$, 
associated with a hyperplane arrangement $\mathcal A$,
as introduced in \cite{Ardila2016},
coincides with the generating function for the Whitney numbers of the second kind 
of the matroid associated with $\mathcal A$, and that there is a natural inclusion
$\operatorname{H}^\bullet(Y_{\mathcal A}) \hookrightarrow \operatorname{IH}^\bullet(Y_{\mathcal A})$.
This argument applies only to realisable matroids, since it relies on the geometry of the (singular) algebraic variety $Y_{\mathcal A}$. 
Later, Braden--Huh--Matherne--Proudfoot--Wang generalised this proof to non-realisable matroids in 
\cite{Braden2020} by algebraically defining analogues of $\operatorname{H}^\bullet(Y_{\mathcal A})$ and $\operatorname{IH}^\bullet(Y_{\mathcal A})$ 
for an arbitrary matroid $M$, namely the graded M\"obius algebra $\operatorname{B}^\bullet(M)$ and the matroid intersection cohomology $\operatorname{IH}^\bullet(M)$ of $M${\color{blue},} respectively.
However, for non-realisable matroids, geometric models for these objects remained unknown.
In this paper, we construct the tropical matroid Schubert variety $Y_M$ and 
show that its tropical cohomology is isomorphic to the graded M\"obius algebra $\operatorname{B}^\bullet(M)$. 

Even when $M$ is realisable over $\mathbb{C}$, this still remains interesting from the perspective of tropical cohomology:
it provides a family of examples of complex algebraic varieties that are not rationally smooth, 
equivalently, whose intersection cohomology complex is not isomorphic to the constant sheaf (\cite[Proposition 8.3]{Fiebig2014}), 
and whose singular cohomology ring nevertheless coincides with the tropical cohomology ring of their tropicalisations.
To the best of our knowledge, 
this appears to be the first result establishing compatibility between 
these cohomology theories when the tropical variety is not regular at infinity in the sense of \cite{Mikhalkin2014} (see \cref{rmkcohomtrop} for more details). 
Our results suggest that compactifications of tropical varieties that are not necessarily regular at infinity may contain interesting geometric features and therefore should be studied further.
\subsection{Results}

Let $M$ be a matroid of rank $d$ with ground set $E$ and let $\Sigma_M^{+}$ denote the augmented Bergman fan of the matroid $M$ in $\mathbb{R}^E$ defined in \cite{MR4477425}.
Let $U_M$ be the support of $\Sigma_M^{+}$ in $\mathbb{R}^E$. 
Let $\mathbb{TP}^1 = \mathbb{R}\cup\{-\infty,+\infty\}$ be the tropical projective line.
Then $U_M$ is a polyhedral subspace in $(\mathbb{TP}^1)^E$.
We define the tropical matroid Schubert variety 
$Y_M$ of $M$ as the closure of $U_M$ in $(\mathbb{TP}^1)^E$. 
We will show that it admits an extended polyhedral structure and hence is an extended polyhedral space in the sense of \cite{Mikhalkin2014, Jell2019, Itenberg2019, Amini2020}.
Moreover, it is a tropical variety, see \cref{deftropvar}.

\smallskip

Denote by $\operatorname{H}^{p,q}(Y_M)$ the $(p,q)$-th tropical cohomology of $Y_M$ with $\mathbb{R}$-coefficients, 
see \cref{tropvar} for the definition. Our first result is the following. 
\begin{theorem}\label{thm1}
  The tropical cohomology of $Y_M$ is concentrated in bidegrees $(p,p)$, that is, 
  for $p \neq q$, $\operatorname{H}^{p,q}(Y_M) = 0$.
\end{theorem}

In the case when $M$ is realisable over $\mathbb{C}$, this is directly analogous to the 
algebraic-geometric situation. If $M$ is realised by a hyperplane arrangement 
$\mathcal A$ over $\mathbb{C}$, the singular cohomology $\operatorname{H}^\bullet (Y_{\mathcal A})$ is pure of Hodge--Tate type, i.e., the cohomology is concentrated in 
even degrees and $\operatorname{H}^{2p}(Y_{\mathcal A})$ is pure of Hodge $(p,p)$-type.
This follows from the following quick argument. 
That $\operatorname{H}^\bullet (Y_\mathcal{A})$ is pure Hodge follows from the fact that affine paved varieties are pure Hodge, see \cite[Theorem 2.1, Theorem 3.1]{Bj_rner2009}.
Then \cite[Section 4]{Huh2017} states that $Y_\mathcal{A}$ has a smooth resolution 
$X_{\overline{\mathcal{A}}} \to Y_\mathcal{A}$ 
by a wonderful compactification $X_{\overline{\mathcal{A}}}$ 
which induces an inclusion
\[
\operatorname{H}^\bullet(Y_\mathcal{A}) \hookrightarrow \operatorname{IH}^\bullet(Y_\mathcal{A}) \hookrightarrow \operatorname{H}^\bullet(X_{\overline{\mathcal{A}}}).
\]
The first inclusion is due to the pure Hodge structure of $\operatorname{H}^\bullet(Y_\mathcal{A})$ and the second comes from the decomposition theorem by \cite{beilinson1982faisceaux}. 
Moreover, the cycle class map of $X_{\overline{\mathcal{A}}}$ gives an isomorphism between its Chow ring and its singular cohomology ring. This implies that $\operatorname{H}^\bullet(X_{\overline{\mathcal{A}}})$ is pure of Hodge--Tate type. 
Since $\operatorname{H}^\bullet({Y_\mathcal{A}})$ injects 
(as a morphism of mixed Hodge structures) into a pure Hodge--Tate structure, 
it is itself pure of Hodge--Tate. 

\medskip

Since $\operatorname{H}^{\bullet,\bullet}(Y_M)$ is concentrated in $(p,p)$-degrees,
we write $\operatorname{H}^\bullet(Y_M)$ as the tropical cohomology of $Y_M$ where odd degrees vanish
and $\operatorname{H}^{2p}(Y_M) = \operatorname{H}^{p,p}(Y_M)$.

\begin{theorem}\label{thm2}
  There is an isomorphism of graded algebras 
  \[\Psi :  \operatorname{H}^{\bullet}(Y_M) \to \operatorname{B}^\bullet(M)\]
  where $\operatorname{B}^p(M)$ is of cohomological degree $2p$.
\end{theorem}

This theorem identifies the tropical cohomology of $Y_M$ and the graded M\"obius algebra, extending
the case of arrangement Schubert varieties for realisable matroids to arbitrary matroids.

\subsection{Idea of the proofs}\label{idea}
We briefly discuss the ideas of proofs of the main theorems \cref{thm1,thm2}.

\cref{thm1} is deduced from the
rank spectral sequence associated to a stratification of $Y_M$ by 
augmented Bergman fans. To describe this stratification, we first introduce the notion of admissible pairs.

An admissible pair $(I,F)$ of $M$ is a pair where $I$ is an independent set and $F$ is a flat such that $I \subseteq F$.
For an admissible pair $(I,F)$, we define its matroid minor $M(I,F)$ of $M$ by first localising at $F$
and then contracting $I$ (see \cref{defadmpair} for details).

\begin{theorem}[\cref{rankstrat} and \cref{stratYM}]\label{thm3}
  $Y_M$ admits a stratification by the support of augmented Bergman fans labelled by admissible pairs
  \[Y_M = \coprod_{(I,F) \text{ admissible}} U_{M(I,F)}.\]
  The partial order on admissible pairs induced by the stratification
   is given by 
   \begin{center}
  $(J,G) \preceq (I,F)$ 
  if and only if 
  $I \subseteq J$ and $ G \preceq F$.
  \end{center}
\end{theorem}

We construct a rank spectral sequence induced by this stratification and use it to prove \cref{thm1}.
See \cref{rankspprop,rankofflatfiltr} for details.
Moreover, we show that $\dim \operatorname{H}^{p,p}(Y_M)=\dim \operatorname{B}^p(M)$ in \cref{rankofflatfiltr,dimeq}.

Since the Whitney numbers of the second kind (first introduced in \cite{MR1503085}) are not palindromic in general, 
the dimensions in \cref{dimeq} need not satisfy Poincaré duality; hence $Y_M$ need not be smooth in any reasonable sense.
To show that there exists an isomorphism of graded algebras as in \cref{thm2}, 
the dimension count alone is not enough
since we have to check if the tropical cup product is compatible with the product in the graded M\"obius algebra.
For this, we take advantage of a tropical analogue of a resolution of singularities.

Viewing $Y_M$ as a closed subvariety of $(\mathbb{TP}^1)^E$, 
we can perform a tropical analogue of successive toric blow-ups of $(\mathbb{TP}^1)^E$ and obtain
$\overline{\Sigma^+_E}=\mathbb{TP}_{\Sigma^+_E}$,
where $\Sigma^+_E$ is the augmented Bergman fan of the Boolean matroid on $E$ (see \cref{matroids}).
More precisely, this is the map induced by the inclusion of $\Sigma^+_E$ into the fan of $(\mathbb{TP}^1)^E$.
Let $\overline{\Sigma^+_M}$ be the canonical compactification of $\Sigma^+_M$.
We show that the following diagram commutes:

\[\begin{tikzcd}
	{\overline{\Sigma^+_M}} & {\overline{\Sigma^+_E}} \\
	{Y_M} & {(\mathbb{TP}^1)^E.}
	\arrow["/"{marking, font=\tiny}, hook, from=1-1, to=1-2]
	\arrow[two heads, from=1-1, to=2-1]
	\arrow[two heads, from=1-2, to=2-2]
	\arrow["/"{marking, font=\tiny}, hook, from=2-1, to=2-2]
\end{tikzcd}\]
By the result of \cite{Amini2024c}, 
the tropical cohomology rings of ${\overline{\Sigma^+_M}}$, ${\overline{\Sigma^+_E}}$, and ${(\mathbb{TP}^1)^E}$
coincide with the Chow rings of the corresponding fans.
We then show that the induced pullback map on cohomology is injective
by using the identification with Chow rings and 
matching the dimensions of the graded pieces.
We denote the pullback map by $\Phi$.

\begin{proposition}\label{propinc}
  The map 
  \[\Phi: \operatorname{H}^{\bullet,\bullet}(Y_M) \hookrightarrow \operatorname{H}^{\bullet,\bullet}(\overline{\Sigma^+_M})\]
  is an injective homomorphism of bigraded algebras.
\end{proposition}
Using this proposition, we complete the proof of \cref{thm2} by identifying the image of $\Phi$ 
with the graded M\"obius algebra.

\subsection*{Outline}
In \cref{prelim}, we review necessary preliminaries that are used throughout the paper. 
In \cref{combwhit}, we give a combinatorial identity that provides a combinatorial understanding 
of the spectral sequence used in \cref{secrankfilt}.
In \cref{geomofYM}, we explain the geometry of $Y_M$ and in particular its stratification by admissible pairs.
In \cref{secrankfilt}, we study the spectral sequences associated with the stratification constructed in \cref{geomofYM} 
to compute the $(p,q)$-th tropical cohomology groups of $Y_M$. 
In \cref{pullbackM}, we show the commutativity of the diagram depicted in \cref{idea}, yielding the identification of the tropical cohomology with the graded M\"obius algebra.
In \cref{comparisonsec}, we show that in the realisable case, $Y_{M}$ coincides with the extended tropicalisation of $Y_\mathcal{A}$.
In \cref{examples}, we present some examples of $Y_M$ and the stratification by admissible pairs.

\subsection*{Conventions and Notations}
Unless otherwise stated, all algebras (including cohomology rings) and vector spaces in this paper
are taken with $\mathbb{R}$-coefficients.
We note however that since $\Sigma_M^+$ is unimodular and saturated in the sense of \cite{Amini2024c}, 
all the main statements also hold with $\mathbb{Z}$-coefficients.
When there is no ambiguity, we may drop \emph{tropical} and \emph{extended}.
We write $\,\preceq\,$ for a partial order and $\precd$ for the covering relation, namely, we write $a\precd b$ if $a\prec b$ and there is no $c$ with $a\prec c \prec b$.
If a poset $\mathcal{P}$ has a unique minimal element, we denote it by $\underline{0}_\mathcal{P}$ or just $\underline{0}$.
For a face $\sigma$ of a polyhedral complex, we write $|\sigma|$ for its dimension.
By a cone, we always mean a closed polyhedral strongly convex cone.

\subsection*{Acknowledgements}
The author thanks Omid Amini for careful reading of the manuscript and for many comments that improved the exposition and organisation of this paper.
The author thanks Lyuhui Wu for a helpful early discussion. 
The author is grateful to Nicholas Proudfoot for helpful correspondence concerning arrangement Schubert varieties.

\section{Preliminaries}\label{prelim}
This section collects the notation and background used throughout the paper.
In \cref{matroids} we fix matroid notation (including minors and free coextensions), 
recall the augmented Bergman fan, and the graded M\"obius algebra.
In \cref{tropvar} we fix the category of extended polyhedral spaces and 
the tropical (co)homology theory we use (in the sense of \cite{Itenberg2019}), 
including cup and cap products, and Poincar\'e duality (in the sense of \cite{Shaw2013a,Gross2023}).

\subsection{Matroids and (augmented) Bergman fans}\label{matroids}
We briefly review matroid theory notions used in this paper. A standard reference is \cite{Oxley2011}.

Let $M$ be a matroid on $E=\{1,\dots,n\}$ of rank $d$.
We write $\mathcal I(M)$ for its independence complex, $L(M)$ for its lattice of flats,
$r_M$ for its rank function, ${cl}_M$ for its closure operator, 
$\chi_M (t)= \sum_k \chi_M^k t^k $ for its characteristic polynomial, 
and $\overline{\chi}_{M} (t) \coloneq \chi_M (t) / (t-1) = \sum_k \overline{\chi}_{M}^k t^k $ for its reduced characteristic polynomial.
We write $\operatorname{rk}(M)$ for its rank as a matroid.
We denote the dual matroid by $M^*$. 
Denote the $i$-th $f$-vector of $M$ by $f_M^i$ which is the number of independent sets of cardinality $i$.

For a flat $F\in L(M)$, let $M^F$ denote the localisation of $M$ to $F$ which is a matroid minor on $F$, and denote by $M_F$ the contraction of $M$ by $F$ which is a matroid minor on $E\setminus F$.
Similarly, for an independent set $I\in\mathcal I(M)$, let $M_I$ denote the contraction of $M$ by $I$ which is a matroid minor on $E \setminus I$.
For a single element $e\in E$, we denote by $M/e$ the contraction of $M$ by $e$.

Let $\widetilde{E} \coloneq E \sqcup \{0\}$.
We write $M+0$ for the free extension of $M$ by an element $0$ which is a matroid on $\widetilde{E}$.
\begin{definition}\label{def:coext}
The \emph{free coextension} of $M$, denoted by $\widetilde{M}$, is the matroid on $\widetilde{E}$
defined by $\widetilde{M}\coloneq(M^*+0)^*$.
\end{definition}
See \cite[7.2]{Oxley2011} for more details about free extensions and coextensions.

\begin{lemma}[{\cite[Remark 6.15.3c]{Brylawski}}]\label{coextf}
  Let $M$ be a matroid of rank $d$ and $\widetilde{M}$ be the free coextension of $M$. 
  Then $|\overline{\chi}^k_{\widetilde{M}}| = f^{d-k}_M$. 
\end{lemma}

\begin{definition}[{\cite{Huh2017}}]
  The \emph{graded M\"obius algebra} of $M$ denoted by $\operatorname{B}^\bullet (M)$
  is the graded vector space 
  \[\operatorname{B}^\bullet (M) = \bigoplus_{F \in L(M)} \mathbb{R}y_F\]
  where the degree of $y_F$ is the rank of $F$ with graded multiplication defined by 
  \[y_F \cdot y_G = \begin{cases}
   y_{F\vee G} & \text{ if } r_M(F\vee G) = r_M(F)+r_M(G)\\
   0 & \text{ else.}
  \end{cases}\]
\end{definition}

Let $(e_i)_{i\in E}$ be the standard basis of $\mathbb{R}^E$.
For a subset $S \subseteq E$, let $e_S \coloneq \sum_{i\in S}e_i$.
We write $\Sigma_M$ for the Bergman fan of $M$, 
following the convention of \cite{Adiprasito2018}, 
namely as a simplicial fan in $\mathbb{R}^E/\langle e_E\rangle$. 
This is a quotient version of the Bergman fan, 
which was first introduced in \cite{Ardila2006}.

\begin{definition}[{\cite{MR4477425}}]\label{def:compair}
  Let $I\in \mathcal{I}(M)$ and $\mathcal{F}$ be a flag of proper flats.  
  A \emph{compatible pair} is a pair $(I,\mathcal{F})$ where $I \subseteq F$ holds for all $F \in \mathcal{F}$.
\end{definition}

\begin{definition}[{\cite{MR4477425}}]
The \emph{augmented Bergman fan} of $M$ denoted by $\Sigma^+_M$ is a simplicial fan in $\mathbb{R}^E$
whose cones are of the form 
\[\sigma_{I,\mathcal{F}} = \operatorname{cone}(e_i)_{i\in I} + \operatorname{cone}(-e_{E\setminus F})_{F \in \mathcal{F}}\]
for each compatible pair $(I,\mathcal{F})$.
\end{definition}

If $M$ is the Boolean matroid on $E$, we simply write $\Sigma^+_E$ for its augmented Bergman fan.

\begin{remark}
 The Bergman fan and the augmented Bergman fan are defined for loopless matroids in the literature.
 However, we extend the definition of these for any matroid as follows. Let 
 $M$ be a matroid on $E$ and let
 $L \subseteq E$ be the set of loops.
 Then $M_L$ is a loopless matroid on $E\setminus L$.
 Consider the inclusion $i : \mathbb{R}^{E\setminus L} \hookrightarrow \mathbb{R}^{E}$.
 Then the Bergman fan and the augmented Bergman fan of $M$ 
 are defined by first defining the corresponding fans for $M_L$ 
 and taking the image induced by $i$.
\end{remark}

For a fan $\Sigma$, let $|\Sigma|$ be the support of $\Sigma$, i.e., $|\Sigma| = \bigcup_{\sigma\in \Sigma}\sigma$.
In \cite{Eur2023}, they explain that the support of 
the augmented Bergman fan of $M$ can be identified with 
the support of the Bergman fan of $\widetilde{M}$. We briefly describe this identification.

Let 
\[\gamma: \mathbb{R}^{\widetilde{E}}/\langle e_{\widetilde{E}} \rangle   \to \mathbb{R}^E\]
be the isomorphism of vector spaces given by
\[[a_0,a_1,\dots , a_n] \mapsto (a_1-a_0, \dots , a_n-a_0).\]

\begin{lemma}[{\cite[Lemma 5.14]{Eur2023}}]\label{coextaugber}
  The map $\gamma$ restricted on $|\Sigma_{\widetilde{M}}|$ has its image as $|\Sigma^+_M|$. 
\end{lemma}

\subsection{Tropical varieties}\label{tropvar}
We recall some notions from polyhedral geometry and tropical geometry.
Good references for this section include \cite{Itenberg2019,Jell2018,Amini2020,Gross2023}.

A \emph{polyhedron} in $\mathbb{R}^n$ is a finite intersection of half-spaces.
A \emph{polyhedral complex} $\mathcal{C}$ in $\mathbb{R}^n$ is a finite set of polyhedra in $\mathbb{R}^n$ such that 
for $P, Q \in \mathcal{C}$, the intersection $P \cap Q \in \mathcal{C}$ and $\mathcal{C}$ is closed under faces.
The support of the polyhedral complex $\mathcal{C}$ denoted by $|\mathcal{C}|$
is the union of all polyhedra in $\mathcal{C}$. We say $\mathcal{C}$ is a \emph{polyhedral structure} of $|\mathcal{C}|$.
For any $S \subseteq \mathbb{R}^n$ such that there exists a polyhedral structure $\mathcal{C}$, we say $S$ is a \emph{polyhedral set}. 
In particular, if $\mathcal{C}$ is a fan, i.e., a polyhedral complex of cones pointed at $\underline{0}$, 
we say $|S|$ is a \emph{fanfold}. 
For a cone $\eta \in \mathcal{C}$, we denote by $\langle \eta \rangle$ the vector subspace spanned by $\operatorname{relint}(\eta)$.
We denote by $\mathcal{C}_d$ the set of $d$-dimensional faces of $\mathcal{C}$.

Let $\mathbb{T}=\mathbb{R}\cup\{+\infty\}$ be equipped with the order topology.
An \emph{extended polyhedron} in $\mathbb{T}^n$ is the closure in $\mathbb{T}^n$ of a polyhedron in $\mathbb{R}^n$.
An \emph{extended polyhedral set} in $\mathbb{T}^n$ is defined analogously to a polyhedral set, 
with polyhedra replaced by extended polyhedra.
For an extended polyhedral set $S$ in $\mathbb{T}^n$, we may define a sheaf of abelian groups 
$\operatorname{Aff}_S$ by
\[U \subseteq S \text{ open } \mapsto \{ f : U \to \mathbb{R} \,\vert\, \text{ continuous locally affine} \}.\]

\begin{definition}[{\cite[Definition 2.2]{Gross2023}}]
  An \emph{extended polyhedral space} consists of a second-countable Hausdorff topological space $X$
  with a sheaf $\operatorname{Aff}_X$, such that for every point $x\in X$
  there exist an open neighbourhood $U\subseteq X$, an open subset $V$ of an extended polyhedron in $\mathbb{T}^n$
  (for some $n\in\mathbb{N}$), and a homeomorphism $\phi:U\to V$ whose pullback identifies the sheaves:
  the map $f\mapsto f\circ \phi$ induces an isomorphism $\phi^{-1}\operatorname{Aff}_V \xrightarrow{\sim} \operatorname{Aff}_U$.
  We call $(U,V,\phi)$ a \emph{chart}.
\end{definition}

\begin{definition}
  A \emph{morphism of extended polyhedral spaces} is a continuous map $f:X\to Y$
  which is \emph{locally affine}, i.e., for every open $V\subseteq Y$ and every
  $\varphi\in\operatorname{Aff}_Y(V)$, the pullback $\varphi\circ f$ lies in $\operatorname{Aff}_X(f^{-1}(V))$.
  Equivalently, $f$ induces a morphism of sheaves $f^{-1}\operatorname{Aff}_Y\to \operatorname{Aff}_X$.
\end{definition}

\begin{definition}
  Let $X$ be an extended polyhedral space. 
  An \emph{extended polyhedral structure} $\mathcal{X}$ on $X$ is given by choosing, for each chart $(U,V,\phi)$, 
  an extended polyhedral structure on $V$ (and hence on $U$ via $\phi$), compatibly on overlaps.  
\end{definition}

Let $\Delta$ be a simplicial fan of dimension $d$ in $ N =\mathbb{R}^n$. Let $N^\vee = \Hom(N, \mathbb{R})$.
Following \cite[Section 2.2]{Amini2024c}, 
we denote by $\mathbb{TP}_\Delta$ the \emph{tropical toric variety} associated with $\Delta$.

The construction is as follows. Let $\mathbb{R}_+$ and $\mathbb{T}_+$ be the non-negative subset of $\mathbb{R}$ and $\mathbb{T}$ respectively.
These are monoids with the addition operator.
Then $\mathbb{T}$ and $\mathbb{T}_+$ can be seen as modules over the
semiring $\mathbb{R}_+$.
For a cone $\eta\in\Delta$ of dimension $|\eta|$, set
\[\eta^\vee \coloneq \{ m \in N^\vee \, \vert \, \langle m,a \rangle \geq 0 \text{ for all } a \in \eta  \}.\]
Then the \emph{tropical affine toric variety} $A_\eta$ of $\eta$ is
\[
A_\eta \coloneq \Hom_{\mathbb{R}_+}(\eta^\vee,\mathbb{T}) = \mathbb{T}^{|\eta|}\times \mathbb{R}^{n-|\eta|}
\]
where the last equality is given by fixing a basis of $\langle \eta \rangle $ by the rays of $\eta$ and extending to $N$.
The space $\mathbb{TP}_\Delta$ is obtained by gluing $A_\eta$
accordingly with the face relation of $\Delta$.
In particular, $\mathbb{TP}_\Delta$ is an extended polyhedral space.

For each $\eta\in\Delta$, define the \emph{tropical torus orbit} labelled by $\eta$ to be
\[
O(\eta)\coloneq\{\infty_\eta\}\times \mathbb{R}^{n-|\eta|}\subseteq A_\eta,
\]
viewed inside $\mathbb{TP}_\Delta$ via the gluing. 
Here $\infty_\eta$ is the distinguished point, the limit of $t v$ as $t\to +\infty$ for any $v\in \operatorname{relint}(\eta)$.
Note that $O(\eta)$ is a vector space with the origin given by $\infty_\eta$.
Then $\mathbb{TP}_\Delta$ admits the orbit decomposition
\[
\mathbb{TP}_\Delta = \coprod_{\eta\in\Delta} O(\eta),
\]
analogous to the torus orbit stratification of a toric variety. 
In particular, $O(\underline{0}) = N$.
We write $V(\eta)$ the closure of $O(\eta)$ in $\mathbb{TP}_\Delta$.

By the identification of $\langle \eta \rangle$ and $\mathbb{R}^{|\eta|}\times \{0\}$ in $A_\eta$, 
we may consider $\eta$ as a subset of $A_\eta$.
Hence $\Delta$ is a subset of $\mathbb{TP}_\Delta$.
The \emph{canonical compactification} of $\Delta$ denoted by $\overline{\Delta}$
is the closure of $\Delta$ in $\mathbb{TP}_\Delta$. 
Let $\overline{\eta}$ be the closure of $\eta$ in $A_\eta$.
Equivalently, 
$\overline{\Delta} = \bigcup_{\eta\in \Delta} \overline{\eta}$.
In particular, if the fan $\Delta$ is complete, i.e., $|\Delta| = N$,
then $\overline{\Delta} = \mathbb{TP}_\Delta$.

Let $U$ be a fanfold in $N$. Let $\Delta$ be a simplicial fan such that $U \subseteq  |\Delta|$.
\begin{definition}
 The \emph{compactification} $Y$ of $U$ by $\Delta$ is the closure of $U$ in $\mathbb{TP}_\Delta$.
\end{definition}
Note that $Y$ is compact since for a fan structure $\Sigma$ of $U$, 
the closure of each maximal cone $\sigma \in \Sigma$ in $\mathbb{TP}_\Delta$ is compact.

Let $\mathcal{Y}$ be a polyhedral structure of $Y$.
For a face $\sigma\in\mathcal{Y}$, we define its \emph{sedentary type} to be the unique cone $\eta\in\Delta$ such that
\[
\operatorname{relint}(\sigma)\subseteq O(\eta)
\]
and write $\operatorname{sed}(\sigma) = \eta$.
In this case, we call $|\eta|$ the \emph{sedentarity} of $\sigma$. 
In fact, sedentary types can be intrinsically defined without an embedding to a tropical toric variety. 

For $\zeta' \preceq \zeta \in \Delta$, denote the projection map $O(\zeta') \twoheadrightarrow O(\zeta)$ by 
$\pi_{\zeta'}^\zeta$ 
which sends those rays of $\zeta$ that do not belong to $\zeta'$ to infinity.
For simplicity, we write $\pi^\zeta$ when ${\zeta'} = \underline{0}$.

We now recall the definition of tropical cohomology for extended polyhedral spaces.
While there is a sheaf-theoretic approach (see \cite{Jell2018, Gross2023}), 
we use the polyhedral complex definition (see \cite{Itenberg2019, Amini2024c}), 
which is the one used throughout this paper. 
For simplicity, we explain only the case of an embedding into a tropical toric variety.

Let $Y$ be a polyhedral space and let $\mathcal{Y}$ be its polyhedral structure.
Let $\mathbf{F}_p$ and $\mathbf{F}^p$ be the $p$-th \emph{multi-tangent} and $p$-th \emph{multi-cotangent spaces} defined as follows.
For $\sigma \in \mathcal{Y}$ of sedentary type $\eta \in \Delta$, 
\[\mathbf{F}_p(\sigma) \coloneq \sum_{ \substack{\delta \succeq \sigma \\ \operatorname{sed}(\delta) = \eta} } \bigwedge^p \langle \delta \rangle \subseteq \bigwedge^p O(\eta), \qquad \mathbf{F}^p(\sigma) \coloneq \Hom(\mathbf{F}_p(\sigma),\mathbb{R}).\]
For $\tau \prec \sigma \in \mathcal{Y}$, 
we have $i_{\tau \prec \sigma} : \mathbf{F}_p(\sigma) \to \mathbf{F}_p(\tau)$ 
and $i^*_{\tau \prec \sigma} : \mathbf{F}^p(\tau) \to \mathbf{F}^p(\sigma)$.
The map $i_{\tau \prec \sigma}$ is induced by 
$\pi_{\operatorname{sed}(\sigma)}^{\operatorname{sed}(\tau)} : O(\operatorname{sed}(\sigma)) \to  O(\operatorname{sed}(\tau))$
where we take it to be the identity map if $\operatorname{sed}(\tau) = \operatorname{sed}(\sigma)$.

For a fixed orientation of faces in $\mathcal{Y}$, let $\operatorname{sgn}(\tau,\sigma) \in \{-1,+1\}$
be the sign of $\tau \prec \sigma$ which is defined by wedging a normal vector going outward on the left (see \cref{lemwedge}).
Then the tropical homology complex $(C_{p,\bullet}(\mathcal{Y}),\partial)$  and the tropical cohomology complex $(C^{p,\bullet}(\mathcal{Y}),d)$
are defined by 
\[C_{p,q}(\mathcal{Y}) \coloneq \bigoplus_{\substack{\sigma \in \mathcal{Y}\\|\sigma|=q \\ \sigma \text{ compact }}} \mathbf{F}_p(\sigma), \qquad C^{p,q}(\mathcal{Y}) \coloneq \bigoplus_{\substack{\sigma \in \mathcal{Y}\\|\sigma|=q \\ \sigma \text{ compact }}} \mathbf{F}^p(\sigma)\]
where the differential $\partial : C_{p,\bullet} \to C_{p,\bullet -1}$ is given by summing $i_{\tau \prec \sigma}$ with appropriate signs, 
and similarly  $d : C^{p,\bullet} \to C^{p,\bullet +1}$ is given by summing $i^*_{\tau \prec \sigma}$ with appropriate signs as well.
The $(p,q)$-th \emph{tropical homology} of $Y$ denoted by $\operatorname{H}_{p,q}(Y)$ is the $q$-th homology of the complex, i.e.,  $\operatorname{H}_{p,q}(Y) = \mathcal{H}_q(C_{p,\bullet}(\mathcal{Y}))$.
Similarly, the $(p,q)$-th \emph{tropical cohomology} of $Y$ denoted by $\operatorname{H}^{p,q}(Y)$ is the $q$-th cohomology of the complex, i.e., $\operatorname{H}^{p,q}(Y) = \mathcal{H}^q(C^{p,\bullet}(\mathcal{Y}))$.
Tropical (co)homology does not depend on the choice of its polyhedral structure.
Note that the chain complex only considers compact faces. 
We may remove the condition of restriction on compact faces and consider all the faces
and define
the \emph{tropical Borel--Moore homology} denoted by $\operatorname{H}_{p,q}^{\mathrm{BM}}(Y)$ and 
the \emph{tropical cohomology with compact supports} denoted by $\operatorname{H}^{p,q}_{\mathrm{c}}(Y)$.

Let $Y$ be a polyhedral space of pure dimension $d$ with its polyhedral structure $\mathcal{Y}$.
Note that 
$C^{\mathrm{BM}}_{d,d}(\mathcal{Y}) =\bigoplus_{\eta \in \mathcal{Y}_d}\mathbf{F}_d(\eta)  =\bigoplus_{\eta \in \mathcal{Y}_d}\bigwedge^d \langle \eta \rangle $.
Let $\nu \in C^{\mathrm{BM}}_{d,d}(\mathcal{Y})$. Then for a choice of generators $\nu_\eta \in \bigwedge^d \langle \eta \rangle \simeq \mathbb{R}$, 
\[\nu = \sum_{\eta \in \mathcal{Y}_d}\omega(\eta) \nu_\eta\]
where $\omega : \mathcal{Y}_d \to \mathbb{R}$.

We define tropical varieties similarly to \cite{Amini2023} as follows.
\begin{definition}\label{deftropvar}
 Let $Y$ be an extended polyhedral space with polyhedral structure $\mathcal{Y}$ and let $\nu_Y = \sum_{\eta \in \mathcal{Y}_d}\omega(\eta) \nu_\eta \in C^{\mathrm{BM}}_{d,d}(\mathcal{Y})$.
 Then $(Y,\nu_Y)$ is a \emph{tropical variety} if $Y$ is of pure dimension $d$, connected, the element $\nu_Y \in \operatorname{H}_{d,d}^{\mathrm{BM}}(Y)$,
 and $\omega$ is nowhere vanishing, i.e., $\omega : \mathcal{Y}_d \to \mathbb{R} \setminus \{0\}$.
 For a tropical variety $(Y,\nu_Y)$, we say $\nu_Y$ is the \emph{fundamental class} of $Y$.
 We say $Y$ admits a \emph{tropical structure} if there exists such $\nu_Y$ for a fixed polyhedral structure.
\end{definition}
The definition of tropical variety does not depend on the choice of a polyhedral structure $\mathcal{Y}$. 
In particular, the existence of such $\nu_Y$ is independent of $\mathcal{Y}$.

Let $\Sigma$ be a tropical fan of dimension $d$ with the fundamental class defined by the weights of $\Sigma$. 
Note that $\operatorname{H}^{p,q}(|\Sigma|)$ is trivial for $q > 0$ since $\Sigma$ only has $\underline{0}_{\Sigma}$ as a compact face.
Moreover, $\operatorname{H}^{p,0}(|\Sigma|) = \mathbf{F}^p(\underline{0}_\Sigma)$. Similarly, $\operatorname{H}_{p,0}(|\Sigma|) = \mathbf{F}_p(\underline{0}_\Sigma)$.

Let $(Y,\nu_Y)$ be a tropical variety of dimension $d$.
The fundamental class $\nu_Y$ induces a \emph{cap product} 
\[- \frown \nu_Y :\operatorname{H}^{p,q}(Y) \to \operatorname{H}^{\mathrm{BM}}_{d-p,d-q}(Y).\]
See \cite{Jell2018,Gross2023} for more details.
We describe the cap product in the case of a tropical fan $\Sigma$. 
For this, we recall the notion of contraction in exterior algebra.
Let $\alpha \in \bigwedge^p N^\vee$ and $\nu \in \bigwedge^d N$.
Then the contraction map
$\kappa_\alpha : \bigwedge^d N \to \bigwedge^{d-p} N$ is characterised as follows.
Let $\langle -,- \rangle_p : \bigwedge^p N \times \bigwedge^p N^\vee \to \mathbb{R}$ be the non-degenerate pairing for a fixed $p$.
Then for all $\beta \in \bigwedge^{d-p} N^\vee$, the equality
$\langle  \kappa_\alpha(\nu), \beta \rangle_{d-p}  = \langle \nu, \alpha \wedge \beta \rangle_{d}$ holds.

Let $\nu = \sum_{\eta\in \Sigma_d} \omega(\eta)\nu_\eta$.
Then the contraction map induces the cap product 
\[\mathbf{F}^p(\underline{0}_\Sigma) \to \operatorname{H}^{\mathrm{BM}}_{d-p,d}(|\Sigma|) \]
by sending
$\alpha \mapsto \sum_{\eta \in \Sigma_d}\omega(\eta) \kappa_\alpha(\nu_\eta) \in \bigoplus_{\eta \in \Sigma_d} \mathbf{F}_{d-p}(\eta)$.
Dually, we get a map
\[\operatorname{H}_{\mathrm{c}}^{d-p,d}(|\Sigma|) \to \mathbf{F}_p(\underline{0}_\Sigma).\]
This map can be understood as a restriction of 
\[\bigoplus_{\eta \in \Sigma_d} \mathbf{F}^{d-p} (\eta) \to \mathbf{F}_p(\underline{0}_\Sigma)\]
by sending $\sum_\eta \alpha_\eta \mapsto (-1)^{p(d-p)} \sum_{\eta} w(\eta) \kappa_{\alpha_\eta}(\nu_\eta)$.
If the support $|\Sigma|$ is a support of a Bergman fan, 
these maps are isomorphisms. In fact, a more general statement is true for tropical varieties locally modelled by Bergman fans.

\begin{definition}
 We say that a tropical variety $Y$ is \emph{matroidal} at $x \in Y$ if there exists an open neighbourhood $U$ of $x$
 such that $U$ is isomorphic to an open set $V$ in $\mathbb{T}^a\times |\Sigma_M|$ for some matroid $M$ and $a\in \mathbb{N}$.
 A tropical variety $Y$ is \emph{locally matroidal} if $Y$ is matroidal at all $x \in Y$.
\end{definition}

\begin{proposition}[{\cite[Theorem 5.3]{Jell2018}}]\label{locmatpd}
  Let $Y$ be a locally matroidal tropical variety. Then 
  the Poincaré duality holds, i.e., 
  the cap product by the fundamental class 
  \[- \frown \nu_Y :\operatorname{H}^{p,q}(Y) \to \operatorname{H}^{\mathrm{BM}}_{d-p,d-q}(Y)\]
  is an isomorphism.
  Dually 
  \[\operatorname{H}^{p,q}_{\mathrm{c}}(Y) \to \operatorname{H}_{d-p,d-q}(Y)\]
  is an isomorphism.
\end{proposition}

\begin{remark}
Locally matroidal tropical varieties can be considered as a tropical analogue of smooth varieties.
In particular, this condition is often referred to as \emph{smooth} in the literature (see, for instance, \cite{Itenberg2019, Gross2023}).
While ``smooth'' is frequently used in this sense, other notions of tropical smoothness also appear, such as quasilinear and homological smoothness; see \cite{Amini2024b}.
For instance, \cite{Amini2023} emphasises that quasilinear fans form a particularly robust class 
(e.g., they are closed under certain operations and satisfy the K\"ahler package), 
and this suggests that broader notions of tropical smoothness may be desirable.
\end{remark}

Moreover, a \emph{cup product} can be defined which gives a bigraded algebra structure to $\operatorname{H}^{\bullet,\bullet}(Y)$ (see \cite{Gross2023}).
In particular, 
\[\operatorname{H}^\bullet(Y) \coloneq \bigoplus_{k\geq 0} \bigoplus_{k=p+q}\operatorname{H}^{p,q}(Y)\] 
is a graded algebra.

\section{Combinatorial identity for Whitney numbers and independence complexes}\label{combwhit}
In this section we establish a combinatorial identity for the Whitney numbers of
the second kind of a matroid $M$. This identity will be used in the proof of
\cref{rankspprop} to compute the Euler characteristic of the ${}_pE_1$-page of the rank
spectral sequence associated to the tropical matroid Schubert variety $Y_M$.
The argument is purely combinatorial and may be of independent interest, 
so we present it separately from other sections.

Let $M$ be a matroid on $E$ with the set of independent sets $\mathcal{I}(M)$,
the rank function $r_M$, and the lattice of flats $L(M)$. Let $W_p(M)$ be the $p$-th \emph{Whitney number of the second kind} 
which is defined by
\[W_p(M)\coloneq \#\{F \in L(M) \,\vert \, r_M(F)=p\}. \]

\begin{definition}\label{defadmpair}
  Let $M$ be a matroid on $E$. 
  \begin{itemize}
    \item A pair $(I,F)$ is \emph{admissible} if $I$ is an independent set and $F$ is a flat such that $I \subseteq F$.
    \item For an admissible pair $(I,F)$, the matroid $M(I,F)$ is a matroid minor of $M$ on $F \setminus I$ obtained by 
   first taking the localisation at $F$ and then taking the contraction by $I$ i.e., $M(I,F) = (M^F)_I$.
    \item The \emph{rank} of $(I,F)$ denoted by $\operatorname{rk}(I,F)$ is the rank of $M(I,F)$ as a matroid, i.e., 
$\operatorname{rk}(I,F) =  r_M(F) - |I|$.
  \end{itemize}
\end{definition}

\begin{remark}
  We warn the reader not to confuse \emph{compatible pairs} in \cref{def:compair} with \emph{admissible pairs}.
  The relation between them is explained in the proof of \cref{rankstrat} and in \cref{rmkadm}.
\end{remark}

For a matroid $M$ of rank $d$ and an integer $0 \leq k \leq d$, let $f_M^k$ be the number of independent sets
of cardinality $k$ in $M$. Note that $f_M = (f_M^0, \dots , f_M^d)$ is 
the $f$-vector of the independence complex of $M$.

Now fix an integer $p \geq 0$. Define
\[n_p^i(M) \coloneq \sum_{\substack{(I,F) \text{ admissible}\\ \operatorname{rk}(I,F) = p+i}}f_{M(I,F)}^i\]
and 
\[N_p(M) \coloneq \sum_{i \geq 0}(-1)^i n_p^i(M).\]

\begin{proposition}\label{whitprop}
 Let $M$ be a matroid and let $N_p(M)$ be as above. Then
 \[N_p(M)  = W_p(M).\] 
\end{proposition}

\begin{proof}
  First note that \[f_{M(I,F)}^i = \#\{S \subseteq F \setminus I \,\vert \, I\sqcup S \in {\mathcal{I}(M)}, |S|=i \}.\]
  This gives
  \begin{align*}
    N_p(M) &=\sum_{i \geq 0}(-1)^i \sum_{\substack{(I,F) \text{ admissible}\\ \operatorname{rk}(I,F) = p+i}}\#\{S \subseteq F \setminus I \,\vert \, I\sqcup S \in {\mathcal{I}(M)}, |S|=i \}\\
            &=\sum_{\substack{(I,F) \text{ admissible},\, S \subseteq  F\setminus I \\ I \sqcup S \in \mathcal{I}(M), \,\operatorname{rk}(I,F) = p+|S|}} (-1)^{|S|}.
  \end{align*}
  Now let $J \coloneq I \sqcup S$. Then we can replace the condition of $S$ in terms of $J$ by $I \subseteq  J \subseteq F$ and $|J| = r_M(F)-p$.
  Hence we get
  \[N_p(M) = \sum_{F\in L(M)} \; \sum_{\substack{(J,F) \text{ admissible}\\ \operatorname{rk}(J,F) = p}}\;\sum_{I \subseteq J}(-1)^{|J|-|I|}.\]
  For fixed $J$, 
  \begin{align*}
    (1-1)^{|J|}=\sum_{I \subseteq J}(-1)^{|J|-|I|} = \begin{cases}
      1 &\text{ if $|J|=0$}\\
      0 & \text{ else}.
    \end{cases}
  \end{align*}
  Since $|J|=r_M(F)-p$, $|J|=0$ if and only if $F$ is of rank $p$. Hence we conclude that
  \[N_p(M) = \sum_{\substack{F\in L(M)\\r_M(F)=p}}1 = W_p(M)\]
  which completes the proof.
\end{proof}

\section{Geometry and stratification of $Y_M$}\label{geomofYM}
In this section, we describe the geometry of $Y_M$ by giving its polyhedral structure. 
In particular, we show that $Y_M$ admits a stratification by augmented Bergman fans
labelled by admissible pairs. 
This leads us to define the rank filtration of $Y_M$ in \cref{secrankfilt}
which is the core tool to compute the cohomology of $Y_M$. 

Before we describe the geometry of $Y_M$, we briefly discuss the induced polyhedral structure of compactifications.

\subsection{Induced polyhedral structure}
Let $Y$ be a compactification of $U$ by $\Delta$.

\begin{definition}\label{compfan}
Let $\Delta$ be a simplicial fan in $\mathbb{R}^n$, and let $\Sigma$ be a fan 
with the support $U\subseteq |\Delta|$. For a cone $\eta\in\Delta$, define the restriction
\[
\Sigma|_\eta \coloneq \{\sigma\in \Sigma \,\vert\, \sigma \subseteq \eta\}.
\]
We say $\Sigma$ is \emph{$\Delta$-compatible} if 
the support $|\Sigma|_\eta| = U\cap\eta$ for all $\eta\in\Delta$. 
\end{definition}

Let $\pi^\zeta : O(\underline{0}) \twoheadrightarrow O(\zeta)$ 
be the projection of torus orbits for 
$\zeta \in \Delta$.
For a $\Delta$-compatible fan $\Sigma$, we can induce a polyhedral structure on $Y$ as follows.

\begin{definition}\label{indpolcom}
  Let $\Sigma$ be a $\Delta$-compatible fan. The \emph{induced polyhedral structure} $\mathcal{Y}$ of $Y$ by $\Sigma$ is the collection of faces
  \[
  \mathcal{Y}=\{\pi^\zeta(\sigma)\mid \operatorname{relint}(\zeta)\cap \sigma\neq\varnothing,\ \sigma\in\Sigma,\ \zeta\in\Delta\}.
  \]
\end{definition}

\begin{remark}
 It is not \emph{a priori} clear if the above definition indeed gives a polyhedral structure on $Y$. 
 Since the required proof is a technical point-set topological argument, 
 we show it in \cref{toriccomp}.
\end{remark}

\begin{lemma}
  Let $M$ be a matroid on $E$.
  Then the fan $\Sigma^+_M$ is $(\Pi^1)^E$-compatible.
\end{lemma}
\begin{proof}
  By \cite[Proposition 3.6 (b)]{Eur2023}, 
  the fan $\Sigma^+_M$ is a subfan of a common refinement of $\Sigma^+_{E}$ and $(\Pi^1)^E$.
  This shows that $\Sigma^+_M$ is $(\Pi^1)^E$-compatible.
\end{proof}

\subsection{Polyhedral structure of $Y_M$}
First we recall the definition of $Y_M$. Let $M$ be a matroid, $E = \{1,\dots,n\}$ be the ground set of $M$.
Let $U_M$ be the support of $\Sigma^+_M$ in $\mathbb{R}^E$. This leads to an inclusion $U_M \hookrightarrow \mathbb{R}^E$.

Denote the fan structure of the tropical projective line $\mathbb{TP}^1$ by $\Pi^1 = \{\underline{0}, \rho^+,\rho^-\}$ 
where we write $\rho^\pm$ for the positive and the negative ray in $\mathbb{R}^1$. 
Then $(\Pi^1)^E$ is the fan of the tropical toric variety $(\mathbb{TP}^1)^E$.
There is an inclusion $\mathbb{R}^E \hookrightarrow (\mathbb{TP}^1)^E$ as well.

\begin{definition}
 The \emph{tropical matroid Schubert variety} $Y_M$ of a matroid $M$ is the topological closure of $U_M$ in $(\mathbb{TP}^1)^E$ under the inclusion
 \[U_M \hookrightarrow \mathbb{R}^E \hookrightarrow (\mathbb{TP}^1)^E.\] 
\end{definition}

\begin{remark}
  The construction of $Y_M$ is intrinsic to the matroid $M$. 
  Indeed, if $M$ on $E$
  and $M'$ on $E'$ are isomorphic as matroids, then any bijection
  $\varphi \colon E \to E'$ inducing an isomorphism $M \simeq M'$ canonically
  induces an isomorphism
  $Y_M \xrightarrow{\sim} Y_{M'}$
  via the corresponding coordinate permutation
  $(\mathbb{TP}^1)^E \xrightarrow{\sim} (\mathbb{TP}^1)^{E'}$. 
\end{remark}

$Y_M$ is compact since $(\mathbb{TP}^1)^E$ is compact.
However, this closure description does not clarify its geometry.
Hence, we need to explicitly understand its polyhedral structure.

We first describe the torus orbit stratification of $(\mathbb{TP}^1)^E$.
Note that a cone $\eta_{J,K} \in (\Pi^1)^E$ is of the form
\[\eta_{J,K} = \operatorname{cone}(e_j)_{j\in J} + \operatorname{cone}(-e_k)_{k\in K}\]
where $J\cap K = \varnothing$ and $J,K \subseteq E$. In particular, the rays of $(\Pi^1)^E$
are of the form 
\begin{align*}
  \rho^+_j \coloneq \eta_{\{j\},\varnothing}\;, \quad
  \rho^-_k \coloneq \eta_{\varnothing,\{k\}}
\end{align*} 
where $1 \leq j,k\leq n$.
Then the torus orbits stratify $(\mathbb{TP}^1)^E$ giving 
\[(\mathbb{TP}^1)^E = \coprod_{\eta_{J,K} \in (\Pi^1)^E} O(\eta_{J,K}).\]
For simplicity, let $O(J,K) \coloneq O(\eta_{J,K}) $. Let $\pi_{J,K}$ be the projection of strata
\[\pi_{J,K}: O(\underline{0})=\mathbb{R}^E\to O(J,K) \subseteq (\mathbb{TP}^1)^E.\]

Now give the induced polyhedral structure of $Y_M$ by $\Sigma^+_M$ as in \cref{indpolcom}.
Recall that a cone $\sigma_{I,\mathcal{F}} \in \Sigma^+_M$ is of the form 
\[\sigma_{I,\mathcal{F}} = \operatorname{cone}(e_i)_{i\in I} + \operatorname{cone}(-e_{E\setminus F})_{F \in \mathcal{F}}\]
where $(I,\mathcal{F})$ is a compatible pair discussed in \cref{matroids}.
We first classify when 
\[\sigma_{I,\mathcal{F}} \cap \operatorname{relint}(\eta_{J,K}) \neq \varnothing.\]

\begin{lemma} Let $\sigma_{I,\mathcal{F}} \in \Sigma^+_M$ and $\eta_{J,K} \in (\Pi^1)^E$.
  Then
  \[\sigma_{I,\mathcal{F}} \cap \operatorname{relint}(\eta_{J,K}) \neq \varnothing\]
  if and only if
  \[ J \subseteq I \text{ and }  E\setminus F = K \text{ for some } F\in \mathcal{F}\]
\end{lemma}

\begin{proof}
  We show the ``if'' direction first. Consider $e_J-e_K$. This is indeed in the relative interior of $\eta_{J,K}$.
  We claim that $e_J-e_K \in \sigma_{I,\mathcal{F}}$. Since $J \subseteq I$, 
  we have $e_J \in \sigma_{I,\mathcal{F}}$ and $-e_K \in \sigma_{I,\mathcal{F}}$ 
  because there exists $F \in \mathcal{F}$ such that $E\setminus F = K$.
  The cone $\sigma_{I,\mathcal{F}}$ is closed under addition and hence we conclude that $e_J-e_K \in \sigma_{I,\mathcal{F}}$.

  For the ``only if'' direction, again $J \subseteq I$ must hold, since otherwise there exists 
  $j \in J\setminus I$ such that 
  all the points in $\operatorname{relint}(\eta_{J,K})$ have positive $j$-th coordinate 
  contrary to the points of $\sigma_{I,\mathcal{F}}$ which have non-positive $j$-th coordinate.
  Now suppose 
  \[\sum_j\lambda_j e_j- \sum_k\mu_ke_k \in \sigma_{I,\mathcal{F}} \cap \operatorname{relint}(\eta_{J,K})\]
  where $\lambda_j, \mu_k > 0$. Then since 
  \[\sum_j\lambda_j e_j-\sum_k\mu_ke_k \in \sigma_{I,\mathcal{F}},\]
  we have
  \[\sum_k \mu_k e_k \in \operatorname{cone}(e_{E\setminus F})_{F\in \mathcal{F}}.\]
  This means there exists $F\in \mathcal{F}$ such that $E\setminus F = K$.
\end{proof}

For a fixed matroid $M$, 
we consider the orbit $O(J,K)$ where $J$ is an independent set and $K=E\setminus F$ for an admissible pair $(J,F)$.

\begin{proposition}\label{rankstrat}
  Let $J \in \mathcal{I}(M)$ and $K = E\setminus F$ be as above. Then $O(J,K) \cap Y_M$ is the support of $\Sigma^+_{M(J,F)}$ defined in $O(J,K)$.
  Moreover, the induced polyhedral structure of $O(J,K) \cap Y_M$ by $\Sigma^+_M$ 
  coincides with $\Sigma^+_{M(J,F)}$.
\end{proposition}

\begin{proof}
 First note that we are identifying $O(J,K) \simeq \mathbb{R}^{E\setminus {J \sqcup K}}$ to define $\Sigma^+_{M(J,F)}$ in $O(J,K)$.
 By abusing notation, we let $e_i \coloneq \pi_{J,K}(e_i) \in O(J,K)$ for $i \in {E\setminus {J \sqcup K}}$.
 Note that $(e_i\,\vert \, i \in {E\setminus {J \sqcup K}})$ is a basis of $O(J,K)$.
 Let $\sigma_{I,\mathcal{F}} \in \Sigma^+_M$ such that $J \subseteq I$ and $F \in \mathcal{F}$.
 Then 
 \begin{align*}
  \pi_{J,K}(\sigma_{I,\mathcal{F}}) &= \pi_{J,K}(\operatorname{cone}(e_i)_{i\in I} + \operatorname{cone}(-e_{E\setminus G})_{G \in \mathcal{F}})\\
  &= \operatorname{cone}(e_i)_{i\in I \setminus J} + \operatorname{cone}(-e_{F\setminus G})_{\substack{G\in \mathcal{F} \\ G \preceq F}}.
 \end{align*}
 Note that for $I \supseteq J$ and $\mathcal{F}\vert_{F} \coloneq \{G \in \mathcal{F} \,\vert \, G \preceq F\}$,
 which is a flag of flats truncated by $F$, the pair $(I\setminus J,\mathcal{F}\vert_F)$ forms a compatible pair in $M(J,F)$
 and vice versa, any compatible pair in $M(J,F)$ is exactly of the form $(I \setminus J,\mathcal{F})$
 where $J \subseteq I$ and $(I,\mathcal{F})$ is compatible in $M$ and $F$ is the maximal element of $\mathcal{F}$. 
 This completes the proof.
\end{proof}

By \cref{rankstrat}, we conclude that $Y_M$ admits a stratification by augmented Bergman fans of $M(I,F)$.

\begin{corollary}
 $Y_M$ is a tropical variety.
\end{corollary}
\begin{proof}
 Let $M$ be a matroid of rank $d$.
 It is clear that $Y_M$ is connected of pure dimension $d$. 
 Let $\mathcal{Y}_M$ be the induced polyhedral structure of $Y_M$ by $\Sigma^+_M$.
 Note that the maximal dimensional faces of $\mathcal{Y}_M$ are the closure of maximal cones of $\Sigma^+_M$.
 Then $\Sigma^+_M$ satisfies the balancing condition by 
 \cite[Proposition 5.2]{Adiprasito2018} and \cref{coextaugber}
 since the existence of balancing condition only depends on the support.
 Then for  $w(\sigma) =1$ and picking integral generators $\nu_\sigma \in \bigwedge^d \langle \sigma \rangle$
 for $\sigma \in \mathcal{Y}_{M,d}$,
 the element $\nu_Y \coloneq \sum_{\sigma \in \mathcal{Y}_{M,d}}\nu_\sigma \in \operatorname{H}_{d,d}^{\mathrm{BM}}(Y_M)$. 
\end{proof}

\begin{corollary}[Stratification of $Y_M$ by admissible pairs]\label{stratYM}
  Let $U_{M(I,F)} =  |\Sigma^+_{M(I,F)}|$ be defined in $O(I,E\setminus F)$ as in \cref{rankstrat}.
  The variety $Y_M$ admits an augmented Bergman fan stratification labelled by admissible pairs
  \[Y_M = \coprod_{(I,F) \text{ admissible}} U_{M(I,F)}\]
  where the partial order of admissible pairs induced by the stratification
   is given by 
   \begin{center}
  $(J,G) \preceq (I,F)$ 
  if and only if 
  $I \subseteq J$ and $ G \preceq F$.
    \end{center}
\end{corollary}

\begin{proof}
 The disjoint union being $Y_M$ directly follows from \cref{rankstrat}. 
 Note that each $\Sigma^+_{M(I,F)}$ is a fan in $O(I,E\setminus F)$.
 Since the closure of $O(I,E\setminus F)$ in $(\mathbb{TP}^1)^E$ is 
 \[V(I,E\setminus F) = \coprod_{\substack{I \subseteq  A,\, E\setminus F \subseteq  B \\ A \cap B = \varnothing}} O(A,B), \]
 the closure of $U_{M(I,F)}$ in $(\mathbb{TP}^1)^E$ is the same as taking closure in 
 $V(I,E\setminus F)$ which is isomorphic to $(\mathbb{TP}^1)^{F\setminus I}$. 
 Observe that this is again a tropical matroid Schubert variety in $V(I,E\setminus F)$,
 identified with $Y_{M(I,F)}$.
 Applying \cref{rankstrat} again, we get that 
 \[Y_{M(I,F)} = \coprod_{\substack{(J,G) \text{ admissible} \\ I \subseteq J, \, G \preceq F}}U_{M(J,G)}\]
 which gives precisely the order we described.
\end{proof}

\begin{remark}\label{rmkadm}
 Note that in the proof of \cref{rankstrat}, in $M(I,F)$, we only consider flags of flats with the maximal flat $F$.
 By this stratification, $U_{M(\varnothing , E)}$ is the maximal stratum in the partial order.
 Then in $\Sigma^+_M$, the condition of a flag of proper flats $\mathcal{F}$ in the definition of compatible pair 
 can be considered as a flag of flats having the maximal flat $E$.
\end{remark}

\begin{remark}
Let $(I,F)$ be an admissible pair. Consider its matroid minor $M(I,F)$.
The lattice of flats of $M(I,F)$
is canonically isomorphic to the interval $[cl_M(I),F]$ of $L(M)$ by sending $G\in [cl_M(I),F]$ to $G\setminus I \in L(M(I,F))$
where $cl_M(I)$ is the closure of $I$.

  Note that if $I,J\in\mathcal{I}(M)$ satisfy ${cl}_M(I)={cl}_M(J)$,
  then the corresponding matroid minors are isomorphic, i.e., $M(I,F)\simeq M(J,F)$.
  However, we emphasise that it is crucial to distinguish these minors in this stratification.
  Indeed, although the minors are isomorphic, they are defined on different ground sets,
  namely $F\setminus I$ and $F\setminus J$, and this distinction will play an essential role.
  Moreover, $M(I,F)$ need not be loopless even if $M$ is loopless.
  Nevertheless, the stratification of $Y_M$ keeps track of loop elements. See \cref{ex2} for an explicit example.
\end{remark}

Now suppose $M$ is not a connected matroid,
i.e., $M = N \oplus O$ for matroids $N$ and $O$.
Let ground sets of $N$ and $O$ be $E$ and $E'$ respectively. 
Then $M$ is a matroid with the ground set $E \sqcup E'$.

\begin{proposition}\label{Y_Mproductdecomp}
 Let $M = N \oplus O$ be as above. Then $Y_M = Y_N \times Y_O$.
\end{proposition}

\begin{proof}
 We first claim that $U_{M} = U_{N} \times U_{O}$.
 Recall that $|\Sigma_M^+|=|\Sigma_{\widetilde{M}}|$, 
 so we may identify $U_M$ with the support of $\Sigma_{\widetilde{M}}$ 
 where $\widetilde{M}$ is the free coextension of $M$. 
 Let $0$ be the adjoined element of free coextensions.
 \cite[Proposition 8.8]{Amini2023} or \cite[Theorem 7.2]{Shaw2023} state that for a matroid $M=P_0(M_1,M_2)$
 which is defined by the parallel connection of $M_1$ and $M_2$ by element $0$, 
 $|\Sigma_{M}| \simeq |\Sigma_{M_1}| \times |\Sigma_{M_2}|$. 
 Hence, in our case, we only need to show that $\widetilde{M} = P_0(\widetilde{N}, \widetilde{O})$,
 i.e., the parallel connection of $\widetilde{N}$ and $\widetilde{O}$ by the adjoined element $0$
 agrees with $\widetilde{M} = \widetilde{N\oplus O}$. 
 This follows from \cite[Proposition 7.1.15 (iii)]{Oxley2011} since 
 $P_0(\widetilde{N},\widetilde{O})/0 = N \oplus O$ implies that 
 $((P_0(\widetilde{N},\widetilde{O})/0)^*+0)^*  = \widetilde{N \oplus O}$ 
 and since $0$ is a cofree element in $P_0(\widetilde{N},\widetilde{O})$,
 $((P_0(\widetilde{N},\widetilde{O})/0)^*+0)^* = P_0(\widetilde{N},\widetilde{O})$.

 Then since the topological closure of $U_{M}$ in ${(\mathbb{TP}^1)}^{E\sqcup E'}$ 
 agrees with the product of the closure of $U_{N}$ in ${(\mathbb{TP}^1)}^{E}$
 and $U_{O}$ in ${(\mathbb{TP}^1)}^{E'}$ respectively, $Y_M = Y_N \times Y_O$ holds.
\end{proof}

\begin{corollary}\label{Y_Mkunneth}
  Let $M=N \oplus O$ be as above. Then
  $\operatorname{H}^{\bullet,\bullet}(Y_M) \simeq \operatorname{H}^{\bullet,\bullet}(Y_N) \otimes \operatorname{H}^{\bullet,\bullet}(Y_O)$.
\end{corollary}

\begin{proof}
 This comes from applying the Künneth formula for tropical cohomology, see \cite[Theorem B]{Gross2023}. 
\end{proof}

\begin{remark}
  Note that there is an isomorphism of graded posets $L(M) \simeq L(N)\times L(O)$ 
  (e.g., see \cite[Proposition 4.2.16]{Oxley2011}).
  Together with \cref{Y_Mproductdecomp,Y_Mkunneth}, this shows that the assignments
  $M \mapsto L(M)$ and $M \mapsto Y_M$ both send direct sums to products.
  This parallel behaviour suggests that there might be a more systematic, possibly functorial,
  relationship among the matroidal combinatorics of $M$, the geometry of $Y_M$, and the graded-poset combinatorics of $L(M)$.
\end{remark}

\section{Cohomology of $Y_M$ via the rank filtration}\label{secrankfilt}
In this section, we compute the tropical cohomology of $Y_M$ by studying a filtration 
induced by the stratification introduced in \cref{geomofYM}. 
In particular, we prove \cref{thm1} by using the spectral sequence induced by this filtration.

Let $M$ be a matroid of rank $d$.
Recall that each stratum is paired with a matroid minor of the form $M(I,F)$ where $(I,F)$ is an admissible pair.

\begin{definition}[Rank filtration of $Y_M$]
Define the \emph{rank of a stratum} $U_{M(I,F)}$ by the rank of $M(I,F)$.
The \emph{rank filtration} of $Y_M$ is the descending filtration
\[\varnothing \subseteq  Y^d \subseteq Y^{d-1} \subseteq \dots \subseteq Y^0 = Y_M \]
where $Y^k$ is the union of strata of rank $\geq k$.
\end{definition}

Clearly $Y^0 =Y_M$ and $\dim (Y^k \setminus Y^{k+1}) = k$. Let $\mathcal{Y}_M$ be the induced polyhedral structure of $Y_M$ by $\Sigma^+_M$.
For a face $\delta \in \mathcal{Y}_M$ we define the \emph{rank} of $\delta$ denoted by 
$\operatorname{rk}(\delta)$ as the rank of the stratum containing $\operatorname{relint}(\delta)$.

Recall that for fixed $p$, $\operatorname{H}^{p,q}(Y_M)$ can be computed by the $q$-th cohomology of $C^{p,\bullet}$
where 
\[C^{p,q} = \bigoplus_{\substack{\tau \in \mathcal{Y}_M,\,|\tau|=q}} \mathbf{F}^p(\tau).\]
We removed the condition of compactness of faces since $Y_M$ itself is compact.
The rank filtration of $Y_M$ induces a filtration on $\operatorname{H}^{p,\bullet}(Y_M)$ as follows.

\begin{lemma}[Rank filtration of $\operatorname{H}^{p,\bullet}(Y_M)$]
 Let 
 \[R^k C^{p,q} = \bigoplus_{\substack{\tau \in \mathcal{Y}_M,\,|\tau|=q \\ \operatorname{rk}(\tau)\geq k}} \mathbf{F}^p(\tau).\]
 Then $R^\bullet$ gives a descending filtration for $C^{p,\bullet}$ and hence for $\operatorname{H}^{p,\bullet}(Y_M)$.
\end{lemma}

\begin{proof}
  Note that the differential $d : C^{p,q} \to C^{p,q+1}$
  either decreases or preserves the sedentarity of faces. 
  Hence it suffices to show that the rank filtration is closed under sedentarity decrease.
  This is clear by \cref{stratYM} since $(J,G) \preceq (I,F)$ implies that $\eta_{I,E\setminus F} \preceq\eta_{J,E\setminus G}$. 
\end{proof}

The filtration $R^\bullet$ induces an associated graded complex
\[\operatorname{gr}^\bullet_R C^{p,\bullet} \coloneq R^\bullet C^{p,\bullet} / R^{\bullet +1}C^{p,\bullet}, \qquad
\operatorname{gr}_R^k C^{p,q} = \bigoplus_{\substack{|\tau|=q\\ \operatorname{rk}(\tau)=k}}\mathbf{F}^p(\tau).\]

For simplicity, for an admissible pair $(I,F)$,
let $\mathbf{F}_i(\underline{0}_{(I,F)})\coloneq \mathbf{F}_i(\underline{0}_{\Sigma^+_{M(I,F)}})$ and similarly for $\mathbf{F}^i$ as well.

\begin{proposition}[Rank spectral sequence]\label{rankspprop}
Let 
\[{}_pE_0^{a,b}(M) \coloneq \operatorname{gr}_R^a C^{p,a+b} = \bigoplus_{\substack{|\tau|=a+b\\ \operatorname{rk}(\tau)=a}}\mathbf{F}^p(\tau) \]
 be the zeroth page of the spectral sequence induced by $R^\bullet$.  
 Then, the following hold.

 \begin{enumerate}[label=(\alph*)]
  \item\label{a} The spectral sequence degenerates at ${}_pE_2$ and vanishes for $b \neq 0$, i.e., 
  \[
    {}_pE_2^{a,0}(M) \simeq \operatorname{H}^{p,a}(Y_M). 
  \]
  \item \label{b} For each $p\ge 0$, the Euler characteristic of
  ${}_pE_1(M)$ is the $p$-th Whitney number of the second kind, that is, 
  \[
    \sum_{a\geq 0}(-1)^{a-p} \dim {}_pE_1^{a,0}(M) = W_p(M).
  \]
\end{enumerate}

\end{proposition}

\begin{proof}
  Note that $d_0$ in ${}_pE_0$ preserves the rank and hence the stratification. 
  This implies that ${}_pE_1$ is equal to the direct sum of the compact support cohomology of each stratum of fixed rank.
  Hence, 
  \[{}_pE^{a,b}_1(M) = \bigoplus_{\operatorname{rk}(I,F)=a} \operatorname{H}^{p,a+b}_{{\mathrm{c}}} (U_{M(I,F)}).\]
  Since $\Sigma^+_{M(I,F)}$ has the same support as $\Sigma_{\widetilde{M(I,F)}}$ by \cref{coextaugber},
  the support $U_{M(I,F)}$ is locally matroidal. We can apply Poincaré duality to each component and get 
  \[{}_pE^{a,b}_1(M) \simeq \bigoplus_{\operatorname{rk}(I,F)=a} \operatorname{H}_{a-p,-b}(U_{M(I,F)}).\]
  Since the homology of a fanfold is trivial except possibly in $(k,0)$-degrees for $k \geq 0$, we conclude that 
  ${}_pE_1^{a,b}$ vanishes for $b \neq 0$. 
  For $b=0$, $\operatorname{H}_{a-p,0}(U_{M(I,F)}) = \mathbf{F}_{a-p}(\underline{0}_{(I,F)})$ holds.
  Moreover ${}_pE_1$ is concentrated in $b=0$. Therefore, it degenerates at ${}_pE_2$ and vanishes for $b\neq 0$.
  This proves \ref{a}.

  For \ref{b}, note that $\dim \mathbf{F}_p(\underline{0}_{\Sigma_M}) = |\overline{\chi}_M^{d-p}|$
  by \cite[Theorem 1.4]{Zharkov2013}
  where $d+1$ is the rank of $M$ and  $\overline{\chi}_M$ is the reduced characteristic polynomial of $M$.
  Since $|\Sigma^+_{M(I,F)}| = |\Sigma_{\widetilde{M(I,F)}}|$, we get 
  \[\dim \mathbf{F}_{a-p}(\underline{0}_{(I,F)}) = \dim \mathbf{F}_{a-p}(\underline{0}_{\Sigma_{\widetilde{M(I,F)}}}) = |\overline{\chi}_{\widetilde{M(I,F)}}^p| = f_{M(I,F)}^{a-p}\]
  where the last equality follows from \cref{coextf}. Note that $\mathbf{F}_{a-p}$ is defined for 
  $a \geq p$ and hence it vanishes for $a < p$. 
  Consider the 
  Euler characteristic of the cochain complex ${}_pE_1^{\bullet,0}$.
  Since
  \[\dim {}_pE_1^{a,0}(M) = \sum_{\operatorname{rk}(I,F)=a} f_{M(I,F)}^{a-p}\]
  we get
  \[\sum_{a\geq p} (-1)^{a-p}\dim {}_pE_1^{a,0}(M) = \sum_{(I,F)} \sum_{a\geq p} (-1)^{a-p} f_{M(I,F)}^{a-p}\]
  which coincides with $N_p(M)$. By \cref{whitprop}, $N_p(M) = W_p(M)$ and this proves \ref{b}.
  This completes the proof.
\end{proof}

  From the proof of \cref{rankspprop}, we already know that $\operatorname{H}^{p,q}(Y_M)=0$ for $q < p$.
  Hence, to prove \cref{thm1}, we only need to show that $\operatorname{H}^{p,q}(Y_M)=0$ for $q > p$. 
  Note that this is equivalent to showing that ${}_pE_1^{\bullet,0}(M)$ is exact at degree $a$ for all $a > p$.
  Consider the following diagram:
\[\begin{tikzcd}
	{\displaystyle\bigoplus_{\operatorname{rk}(I,F)=a} \operatorname{H}^{p,a}_{{\mathrm{c}}} (U_{M(I,F)})} & {\displaystyle\bigoplus_{\operatorname{rk}(I,F)=a+1} \operatorname{H}^{p,a+1}_{{\mathrm{c}}} (U_{M(I,F)})} \\
	{\displaystyle\bigoplus_{\operatorname{rk}(I,F)=a}\mathbf{F}_{a-p}(\underline{0}_{(I,F)})} & {\displaystyle\bigoplus_{\operatorname{rk}(I,F)=a+1}\mathbf{F}_{a+1-p}(\underline{0}_{(I,F)}).}
	\arrow["d^a_1", from=1-1, to=1-2]
	\arrow["{{PD^{*}}}^{-1}", from=2-1, to=1-1]
	\arrow["{PD^{*}}", from=1-2, to=2-2]
	\arrow["\varphi^a", from=2-1, to=2-2]
\end{tikzcd}\]
  Write $\varphi^\bullet \coloneq {PD^{*}} \circ d^\bullet_1 \circ {PD^{*}}^{-1}$ where ${PD^{*}}$ is the dual of the Poincaré duality map on each direct summand.
  More precisely, we apply Poincaré duality using the canonical fundamental class 
  where the top-dimensional cones carry the constant weight $1$.
  Note that this is unique up to sign. For more details on this map, see \cref{lemwedge}.
  Let $\Sigma^+_{M(J,G)}$ be of rank $a$ and 
  $\Sigma^+_{M(I,F)}$ be of rank $a+1$ such that $(J,G) \precd (I,F)$. 
  Now consider $\varphi^a$ restricted on $\mathbf{F}_{i}(\underline{0}_{(J,G)})$ projected to $\mathbf{F}_{i+1}(\underline{0}_{(I,F)})$.
  Denote this map by $\xi$.
  Since $\operatorname{rk}(I,F) - \operatorname{rk}(J,G) = 1$, either $I=J$ and $G\precd F$; 
  or
  $J = I \sqcup j  \in \mathcal{I}(M)$ for some $j\in F \setminus I$ and $F=G$.
  
  \begin{lemma}\label{lemwedge}
  The map $\xi : \mathbf{F}_{i}(\underline{0}_{(J,G)}) \to \mathbf{F}_{i+1}(\underline{0}_{(I,F)})$ sends 
  $u \mapsto \mathbf{n} \wedge u$ where 
  \[\mathbf{n} = \begin{cases}
    -e_j & \text{ if } F=G\\
    e_{F\setminus G} & \text{ if } I=J.
  \end{cases}\]  
  \end{lemma}
  \begin{proof}
    First, since $d : C^{p,q} \to C^{p,q+1}$ splits 
    by preserving rank or increasing rank, $d_1$ in $(_{p}E_1^{\bullet,0},d_1)$ 
    is induced by $d$ that increases rank. Since ${}_pE^{a,b}_1$ is trivial for $b \neq 0$, 
    $d_1$ is nothing but a restriction of $d : C^{p,q} \to C^{p,q+1}$ that increases rank. 
    Hence, $d_1$ is a restriction of 
    \[\bigoplus_{\substack{|\tau|=a+b\\ \operatorname{rk}(\tau)=a}}\mathbf{F}^p(\tau) \to \bigoplus_{\substack{|\sigma|=a+b+1\\ \operatorname{rk}(\sigma)=a+1\\ \tau \prec \!\cdot \sigma}}\mathbf{F}^p(\sigma) .\]
    We may fix an integral generator of $\bigwedge^{|\tau|}\langle \tau \rangle $ by $\nu_\tau$  which is unique up to sign.
    Let $\operatorname{sgn}(\tau,\sigma) \in \{-1,+1\}$ be the sign of $\tau \precd \sigma$ defined by  
    $\nu_\sigma = {\operatorname{sgn}(\tau,\sigma)} \mathbf{n}\wedge \nu_\tau$ where $\mathbf{n}$ is defined by each stratum of $\tau$ and $\sigma$ as above.
    Recall that the dual of the Poincaré duality map is the dual of $\alpha \mapsto \sum_\tau\kappa_{\alpha}(\nu_\tau)$.
    This map sends $\sum_\tau \alpha_\tau \mapsto \lambda \sum_\tau \kappa_{\alpha_\tau}(\nu_\tau) \in \mathbf{F}_{a-p}(\underline{0}_{(J,G)})$
    for $\lambda = (-1)^{p(a-p)}$.

    Then we get the following diagram:
    \[
\begin{tikzcd}
  {\sum_\tau\alpha_\tau \in}
  &[-15mm]
  \begin{array}{c}
    \bigoplus_{\substack{|\tau|=a+b\\ \operatorname{rk}(\tau)=a\\ \tau \in \Sigma^+_{M(J,G)}}}\mathbf{F}^p(\tau)
  \end{array}
  &
  \begin{array}{c}
    \bigoplus_{\substack{|\sigma|=a+b+1\\ \operatorname{rk}(\sigma)=a+1\\ \sigma \in \Sigma^+_{M(I,F)}}}\mathbf{F}^p(\sigma)
  \end{array}
  &[-15mm]
  {\ni \sum_\sigma {\operatorname{sgn}(\tau,\sigma)}\alpha_\tau \circ \pi_{\tau \prec \!\cdot \sigma}}
  \\
  {\lambda\sum_\tau\kappa_{\alpha_\tau}(\nu_\tau) \in}
  &[-15mm]
  {\mathbf{F}_{a-p}(\underline{0}_{(J,G)})}
  &
  {\mathbf{F}_{a+1-p}(\underline{0}_{(I,F)})}
  &[-15mm]
  {\ni \lambda(-1)^p \sum_\sigma {\operatorname{sgn}(\tau,\sigma)} \kappa_{\alpha_\tau \circ \pi_{\tau \prec \!\cdot \sigma}}(\nu_\sigma)}
  \arrow[maps to, from=1-1, to=1-4, bend left=30]
  \arrow[maps to, from=1-1, to=2-1]
  \arrow[from=1-2, to=1-3]
  \arrow[from=1-2, to=2-2]
  \arrow[from=1-3, to=2-3]
  \arrow[maps to, from=1-4, to=2-4]
  \arrow[maps to, from=2-1, to=2-4, bend right=15]
  \arrow[from=2-2, to=2-3]
\end{tikzcd}
\]
where $\pi_{{\tau \prec \!\cdot \sigma}} : \mathbf{F}_p(\sigma) \twoheadrightarrow \mathbf{F}_p(\tau)$ 
is the projection induced by $\pi^{(J,G)}_{(I,F)}$ and $\kappa_\alpha(\nu)$ is the contraction of $\alpha$ by $\nu$.
Then $\xi$ is the map as desired if the diagram commutes, i.e.,  
$\kappa_{\alpha_\tau \circ \pi_{\tau \prec \!\cdot \sigma}}(\nu_\sigma) = (-1)^p \mathbf{n}\wedge \kappa_{\alpha_\tau}(\nu_\tau)$ holds.
Note that 
$\pi^{(J,G)}_{(I,F)}: \langle \sigma \rangle \twoheadrightarrow \langle \tau \rangle $ sends $\mathbf{n}$ to $0$.
Hence, we get 
\[\kappa_{\alpha_\tau \circ \pi_{\tau \prec \!\cdot \sigma}}(\mathbf{n}\wedge\nu_\tau)  
= (-1)^p \mathbf{n}\wedge \kappa_{\alpha_\tau}(\nu_\tau)\]
which completes the proof.
\end{proof}

We have the following cochain complex
  \[
  {\cdots}
  \!\to \!
	{\displaystyle\bigoplus_{\operatorname{rk}(I,F)=a-1}\mathbf{F}_{a-p-1}(\underline{0}_{(I,F)})}
  \!\xrightarrow{\varphi}
  {\displaystyle\bigoplus_{\operatorname{rk}(I,F)=a}\mathbf{F}_{a-p}(\underline{0}_{(I,F)})} 
  \!\xrightarrow{\varphi}
  {\displaystyle\bigoplus_{\operatorname{rk}(I,F)=a+1}\mathbf{F}_{a-p+1}(\underline{0}_{(I,F)})}
  \!\to
  {\cdots}
	\]
which is isomorphic to ${}_pE_1^{\bullet,0}$ by Poincar\'e duality described in \cref{lemwedge}.
Denote this cochain complex by ${}_pD^\bullet$.
We introduce an additional filtration $W^\bullet$ on ${}_pD^\bullet$ as follows.
Let 
\[W^k {}_pD^a \coloneq {\bigoplus_{\substack{\operatorname{rk}(I,F)=a \\ {r}_M(F) \geq k}}\mathbf{F}_{a-p}(\underline{0}_{(I,F)})} \]
be the \emph{flat-rank filtration} on ${}_pD^\bullet$. 
This is indeed a descending filtration of ${}_pD^\bullet$ and hence induces a spectral sequence
${}_p{\Xi}^{\bullet,\bullet}$ where its zeroth page is given by
\[ {}_p{\Xi}_0^{k,l} = \operatorname{gr}_W^{k} {}_p D^{k+l} =  {\bigoplus_{\substack{\operatorname{rk}(I,F)=k+l \\ {r}_M(F) = k}}\mathbf{F}_{k+l-p}(\underline{0}_{(I,F)})}.\]
Then this spectral sequence abuts to $\mathcal{H}^{k+l}({}_p D^\bullet) = \operatorname{H}^{p,k+l}(Y_M)$.

\begin{proposition}\label{rankofflatfiltr}
 ${}_p {\Xi}^{\bullet,\bullet}$ degenerates at ${}_p {\Xi}_1$ and moreover
 \[{}_p {\Xi}_1^{k,l} \simeq \begin{cases} 
  \mathbb{R}^{W_p(M)}  & \text{ if } k=p, \, l=0\\
  0 & \text{ else.}
 \end{cases}\] 
\end{proposition}

\begin{corollary}\label{dimeq}
  $\dim \operatorname{H}^{p,p}(Y_M) = W_p(M)=\dim \operatorname{B}^p(M)$.
\end{corollary}
\begin{proof}
 This directly follows from \cref{rankofflatfiltr}.
\end{proof}

\begin{proof}[Proof of \cref{rankofflatfiltr}]
  For $\mathbf{F}_{k+l-p}(\underline{0}_{(I,F)})$, let $S \subseteq F \setminus I$ such that $S \sqcup I \in \mathcal{I}(M)$
  and $|S| = k+l-p$. Note that
  \[\dim \mathbf{F}_{k+l-p}(\underline{0}_{(I,F)}) = 
  f^{k+l-p}_{M(I,F)} = \#\{S \subseteq F \setminus I \,\vert \, I\sqcup S \in {\mathcal{I}(M)}, |S|=k+l-p \}.\]
  In particular, $S \in \mathcal{I}(M(I,F))$ and hence $\sigma_{S,\varnothing} \in \Sigma^+_{M(I,F)}$.
  This implies that 
  \[w_S \coloneq \bigwedge_{s\in S} e_s \in \mathbf{F}_{k+l-p}(\underline{0}_{(I,F)}) \subseteq \bigwedge^{k+l-p} O(I,E\setminus F)\]
  where the wedge of $e_s$ for $s\in S$ is defined by fixing a total order on $E$.
  Clearly 
  $(w_S \, \vert \, S \in \mathcal{I}(M(I,F)), \,  |S|=k+l-p)$ 
  is a linearly independent family in $\mathbf{F}_{k+l-p}(\underline{0}_{(I,F)})$.
  This shows that $(w_S)$ is a basis of $\mathbf{F}_{k+l-p}(\underline{0}_{(I,F)})$.
  Let $\mathbf{F}_{k+l-p}(\underline{0}_{(I,F)})^{(J)} \coloneq \langle w_S \rangle $
  be the $1$-dimensional subspace spanned by $w_S$
  where $I \sqcup S =J$. Note that this subspace does not depend on the choice of 
  the order of $E$.
  This allows us to decompose $\mathbf{F}_{k+l-p}(\underline{0}_{(I,F)})$ labelled by $J$, that is,
  \[\mathbf{F}_{k+l-p}(\underline{0}_{(I,F)}) = \bigoplus_{J} \mathbf{F}_{k+l-p}(\underline{0}_{(I,F)})^{(J)}.\]
  
  Now fix a flat $F$ of rank $k$. Let $(J,F)$ be an admissible pair such that $|J|=k-p$.
  Let $J_i \subseteq J$ such that $|J_i| = i$.
  Consider the following cochain complex

  \begin{align*}
    0 \to {} &
    \mathbf{F}_0(\underline{0}_{(J,F)})^{(J)}
    \to
    \bigoplus_{J_1}\mathbf{F}_1(\underline{0}_{(J\setminus J_1,F)})^{(J)}
    \to \cdots \\
    & \cdots \to
    \bigoplus_{J_i}\mathbf{F}_i(\underline{0}_{(J\setminus J_i,F)})^{(J)}
    \to \cdots \to
    \mathbf{F}_{|J|}(\underline{0}_{(\varnothing,F)})^{(J)}
    \to 0.
  \end{align*}

  where the differential map is $d = -e_J \wedge -$ with a slight abuse of notation. 
  Denote this complex as ${}_p D^\bullet(J,F)$ where the cohomological degree of $\mathbf{F}_0(\underline{0}_{(J,F)})$ is $p$.
  This complex is isomorphic to the following Koszul complex
  \[ 0 \to \bigwedge^0 V\to \bigwedge^1 V \to \bigwedge^2 V \to \cdots \to \bigwedge^{|J|} V \to 0 \]
  with the differential map $d = v \wedge -$ for $V = \mathbb{R}^J$ and $v = -e_J = -\sum_{j\in J} e_j$.
  Note that the Koszul complex is acyclic if and only if $v$ is non-zero.
  Then, by \cref{lemwedge},
  \[{}_p {\Xi}_0^{k,\bullet} = \bigoplus_{\substack{(J,F) \text{ admissible }\\{r}_M(F)= k, |J| = k-p}} {}_p D^{k+\bullet} (J,F),\] 
  and each ${}_p D^\bullet (J,F)$ is acyclic for all $J$ if and only if  
   $J \neq \varnothing$.
  Hence 
  \[{}_p {\Xi}_1^{k,l} = \begin{cases} 
    \displaystyle\bigoplus_{{r}_M(F) = k} \mathbf{F}_0(\underline{0}_{(\varnothing,F)})  & \text{ if } (k,l) = (p,0)\\
    0 & \text{ else}
 \end{cases}\] 
 which completes the proof.
\end{proof}

\begin{proof}[Proof of \cref{thm1}]
  By \cref{rankofflatfiltr}, 
  $\operatorname{H}^{p,k+l} (Y_M)= {}_p {\Xi}_1^{k,l}$ and
  ${}_p {\Xi}_1$ is concentrated in $k=p, l=0$ which implies that
  $\operatorname{H}^{p,k+l}(Y_M)$ is concentrated in bidegrees $k=p, l=0$ i.e., in $(p,p)$-degrees.
\end{proof}

\begin{remark}
  The proof of \cref{whitprop} can be viewed as a combinatorial shadow of the acyclic Koszul complexes that appear in the proof of \cref{rankofflatfiltr}.
More precisely, the vanishing of the alternating sum for $J\neq\varnothing$ in the proof of \cref{whitprop} mirrors the exactness of the complex ${}_pD^\bullet(J,F)$ for $J\neq\varnothing$ in the proof of \cref{rankofflatfiltr}.
\end{remark}

\section{Pullback to the canonical compactification of $\Sigma_M^+$}\label{pullbackM}
In this section, we study a tropical analogue of successive toric blow-ups of 
$(\mathbb{TP}^1)^E$ to prove the injectivity in \cref{propinc}:
\[
\Phi: \operatorname{H}^{\bullet,\bullet}(Y_M)\hookrightarrow \operatorname{H}^{\bullet,\bullet}(\overline{\Sigma_M^+}).
\]
Moreover,  we show that its image is isomorphic to the graded M\"obius algebra $\operatorname{B}^\bullet(M)$, yielding an isomorphism of graded algebras.
This allows us to complete the proof of \cref{thm2}.

Let $X_E \coloneq \mathbb{TP}_{\Sigma^+_E}$ and $Y_E \coloneq \mathbb{TP}_{(\Pi^1)^E} = (\mathbb{TP}^1)^E$. 
These are canonical compactifications of $\Sigma^+_E$ and $(\Pi^1)^E$ respectively 
since the corresponding fans are complete fans. 
Moreover, $Y_E$ is the tropical matroid Schubert variety of the Boolean matroid on $E$
since the corresponding augmented Bergman fan is complete.

Note that for any matroid $M$ on $E$,
$\Sigma^+_M$ is a subfan of $\Sigma^+_E$.
Also note that $\Sigma^+_E$ can be obtained by applying stellar subdivisions from $(\Pi^1)^E$ by \cite[Proposition 3.6 (b)]{Eur2023}.
Hence we can consider the inclusion map $ \Sigma^+_E \hookrightarrow (\Pi^1)^E$. This induces a map of tropical toric varieties
$\pi' : X_E \to Y_E$.
The map $\pi'$ is a map analogous to successive toric blow-ups of toric varieties.
Denote $X_M \coloneq \overline{\Sigma_M^+}$, the canonical compactification of $\Sigma^+_M$.

\begin{proposition}\label{commdiagresol}
 The following diagram commutes
  \[\begin{tikzcd}
	{X_M} & {X_E} \\
	{Y_M} & {Y_E}
	\arrow["/"{marking, font=\tiny}, "i'" ,hook, from=1-1, to=1-2]
	\arrow[ "\pi"',from=1-1, to=2-1]
	\arrow[ "\pi'",from=1-2, to=2-2]
	\arrow["/"{marking, font=\tiny}, "i"',hook, from=2-1, to=2-2]
\end{tikzcd}\]
where $i$ and $i'$ are closed embeddings as closed subvarieties 
and $\pi$ is the restriction of $\pi'$. 
\end{proposition}

\begin{proof}
  We show that $\pi'(X_M) = Y_M \subseteq Y_E$.
  First note that $\pi'$ is an isomorphism on $U$ for
  $U \coloneq \pi'^{-1}(U_M)$. Then $X_M$ is the closure of $U$ in $X_E$.
  If $\pi'$ is surjective and closed, then $\pi'$ commutes with taking closure and hence we are done.
  The map $\pi'$ is proper and hence closed since $X_E$ is compact and Hausdorff. 
  Surjectivity can be shown by noting that subdivision of a fan gives a surjective morphism of toric varieties.
\end{proof}

\begin{proof}[Proof of \cref{propinc,thm2}]
  Consider the following commuting diagram of bigraded algebras
\[\begin{tikzcd}
	{\operatorname{H}^{\bullet,\bullet}(X_M)} & {\operatorname{H}^{\bullet,\bullet}(X_E)} \\
	{\operatorname{H}^{\bullet,\bullet}(Y_M)} & {\operatorname{H}^{\bullet,\bullet}(Y_E)}
	\arrow["{i'^*}"', from=1-2, to=1-1]
	\arrow["{\pi^*}", from=2-1, to=1-1]
	\arrow["{\pi'^*}"', from=2-2, to=1-2]
	\arrow["{i^*}", from=2-2, to=2-1]
\end{tikzcd}\]  
induced by \cref{commdiagresol}. 

Theorem 1.3 in \cite{Amini2024c} states that
for a simplicial fan $\Sigma$ satisfying tropical Poincar\'e duality,
the tropical cohomology of its canonical compactification is concentrated in $(p,p)$-degrees;
and moreover there is a canonical isomorphism 
$\operatorname{H}^{\bullet,\bullet}(\overline{\Sigma}) \simeq \operatorname{A}^\bullet (\Sigma)$
where $\operatorname{A}^\bullet(\Sigma)$ is the Chow ring of the fan $\Sigma$. 

Note that $X_M$, $X_E$, and $Y_E$ are canonical compactifications of fans $\Sigma^+_M$, $\Sigma^+_E$, and $(\Pi^1)^E$
and hence
the tropical cohomology rings of $X_M$, $X_E$, and $Y_E$ 
can be identified as 
the Chow rings of toric varieties of $\Sigma^+_M$, $\Sigma^+_E$, and $(\Pi^1)^E$.
Moreover, for the toric morphisms appearing in our diagram 
(induced by refinements and inclusions of subfans), this identification is compatible with pullback; this follows 
by tracing the cochain-level description of $\Psi$ in \cite[Section 4.2]{Amini2024c}.

Hence the diagram can be understood in terms of Chow ring of fans as follows:
\[\begin{tikzcd}
	{\operatorname{A}^{\bullet}(\Sigma_M^+)} & {\operatorname{A}^{\bullet}(\Sigma_E^+)} \\
	{\operatorname{H}^{\bullet,\bullet}(Y_M)} & {\operatorname{A}^{\bullet}((\Pi^1)^E)}.
	\arrow["{i'^*}"', from=1-2, to=1-1]
	\arrow["{\pi^*}", from=2-1, to=1-1]
	\arrow["{\pi'^*}"', from=2-2, to=1-2]
	\arrow["{i^*}", from=2-2, to=2-1]
\end{tikzcd}\]

Note that 
\[{\operatorname{A}^{\bullet}((\Pi^1)^E)} = {\operatorname{A}^{\bullet}((\mathbb{P}^1_\mathbb{C})^E)} \simeq \mathbb{R}[y_1, \dots , y_n]/(y_i^2 \,\vert \, i \in E)\] 
is generated by 
$y_i \in {\operatorname{A}^{1}((\mathbb{P}^1_\mathbb{C})^E)}$  where $y_i$ is 
the  pullback of the hyperplane class for the $i$-th $\mathbb{P}^1_\mathbb{C}$. Now consider a subdiagram generated by $y_i$'s:

\[\begin{tikzcd}
	{\operatorname{B}^{\bullet}(M)} & {S^\bullet} \\
	{R^\bullet} & {\operatorname{A}^{\bullet}((\Pi^1)^E)}
	\arrow[two heads, from=1-2, to=1-1]
	\arrow[two heads, from=2-1, to=1-1]
	\arrow["\simeq"', from=2-2, to=1-2]
	\arrow[two heads, from=2-2, to=2-1]
\end{tikzcd}\]
where $R^\bullet$ and $S^\bullet$ are images of $i^*$ and $\pi'^*$ respectively.
Then since $\pi'$ is a sequence of blow-ups, 
the pullback map $\pi'^*$ on Chow rings is injective (see e.g., \cite[Section 6.7]{Fulton1984}).
Hence ${\operatorname{A}^{\bullet}((\Pi^1)^E)} \to S^\bullet$ is an isomorphism.
By \cite[Proposition 2.15]{MR4477425}, the subalgebra generated by the pullback of $y_i$'s
in $\operatorname{A}^{\bullet}(\Sigma^+_M)$ 
is the graded M\"obius algebra $\operatorname{B}^\bullet (M)$. 
Since $R^\bullet \to \operatorname{B}^\bullet (M)$ is surjective, 
the map $\pi^* : {\operatorname{H}^{\bullet,\bullet}(Y_M)} \to \operatorname{B}^\bullet (M)$ is an isomorphism of graded algebras
provided that \[\dim \operatorname{H}^{p,p}(Y_M) =\dim \operatorname{B}^p(M)\]
for all $p$. Then \cref{dimeq} completes the proof of \cref{propinc} by noting $\Phi = \pi^*$.
This proves \cref{thm2}.
\end{proof}

\section{Comparison with arrangement Schubert varieties}\label{comparisonsec}
In this section we compare the tropical matroid Schubert variety $Y_M$ with the arrangement Schubert variety
$Y_{\mathcal A}$ associated to a realisation $\mathcal A$ of $M$.
We consider extended tropicalisation over the trivially valued field $\mathbb{C}$ (in the sense of \cite{Kajiwara2008, Payne2009}).
Using extended tropicalisation, we show that the extended tropicalisation
of $Y_{\mathcal A}$ coincides with $Y_M$.

We first recall the definition of the arrangement Schubert variety $Y_{\mathcal{A}}$ of a hyperplane arrangement $\mathcal{A}$
which was first studied in \cite{Ardila2016}.
For simplicity, we assume that the arrangement is over $\mathbb{C}$.
Let $V$ be a $d$-dimensional $\mathbb{C}$-vector space. 
Let $E \subseteq V$ be a spanning family of cardinality $n$.
This induces an essential arrangement of hyperplanes $\mathcal{A}$ in $V^*$.
Let $i: V^* \to \mathbb{C}^E$ be $\ell \mapsto (\ell(e))_{e\in E}$.

\begin{lemma}
 The map $i$ is an embedding. 
\end{lemma}

\begin{proof}
This follows directly from the condition of $E$ spanning $V$.
\end{proof}

Consider the following inclusion $\mathbb{C} \hookrightarrow \mathbb{P}^1_\mathbb{C}$ induced by 
$\mathbb{C} \hookrightarrow \mathbb{C} \cup \{\infty\}$, $z \mapsto z$.
This induces an open immersion $j : \mathbb{C}^E \hookrightarrow (\mathbb{P}^1_\mathbb{C})^E$.

\begin{definition}
 The \emph{arrangement Schubert variety} $Y_\mathcal{A}$ associated with an arrangement $\mathcal{A}$ 
 is the Zariski closure of the image of $j \circ i$ in $(\mathbb{P}^1_\mathbb{C})^E$.
\end{definition}

Let $\mathbf{a} \in \mathbb{C}^E$ be a generic element. 
The translation map
\[
+\mathbf{a} : \mathbb{C}^E \to \mathbb{C}^E, \qquad \mathbf{x} \mapsto \mathbf{x}+\mathbf{a}
\]
extends to an automorphism of $(\mathbb{P}^1_\mathbb{C})^E$. 
Hence the Zariski closure of the image of
\[
j \circ +\mathbf{a} \circ i
\]
in $(\mathbb{P}^1_\mathbb{C})^E$ is isomorphic to $Y_\mathcal{A}$. 
We fix this translated embedding of $Y_\mathcal{A}$ into $(\mathbb{P}^1_\mathbb{C})^E$.

We may therefore consider the extended tropicalisation of $Y_\mathcal{A}$ with respect to this embedding. 
The following proposition shows that it agrees with our definition of tropical matroid Schubert variety when the matroid is realisable.

\begin{proposition}\label{exttrop}
 Let $M(\mathcal{A})$ be the realisable matroid of $\mathcal{A}$. 
 Then $Y_{M(\mathcal{A})}$ is the extended tropicalisation of $Y_\mathcal{A}$ as a closed subvariety of $(\mathbb{P}^1_\mathbb{C})^E$.
\end{proposition}

\begin{proof}
 This is a consequence of the definition of $Y_{M(\mathcal{A})}$, and
\cite[Proposition 5.13]{Eur2023}, and \cite[Theorem 6.2.18]{Maclagan2015}.
 More precisely, \cite[Proposition 5.13]{Eur2023} states that the tropicalisation of $j \circ+\mathbf{a} \circ i (V^*)$ by the maximal torus of $(\mathbb{P}^1_{\mathbb{C}})^E$
 is $U_{M(\mathcal{A})}$; 
 and \cite[Theorem 6.2.18]{Maclagan2015} states that 
 the extended tropicalisation commutes with taking closure in toric varieties.
\end{proof}

\begin{corollary}
 The tropical cohomology of the extended tropicalisation of $Y_\mathcal{A}$ is isomorphic to the singular cohomology of $Y_\mathcal{A}$.
\end{corollary}

\begin{proof}
  This follows from \cref{exttrop}, \cref{thm2}, and \cite[Theorem 14]{Huh2017} which states that 
  the graded M\"obius algebra is isomorphic to the singular cohomology of $Y_\mathcal{A}$.
\end{proof}

Since $Y_\mathcal{A}$ is generally not rationally smooth (for the smallest example, see \cite[Example 7.2]{Braden2025}),
this gives a family of examples of tropical cohomology coinciding with singular cohomology even when the variety is
not rationally smooth.

\begin{remark}\label{rmkcohomtrop}
    There are several comparison results showing that, under suitable hypotheses, tropical cohomology recovers the singular cohomology of a complex algebraic variety.
    For subvarieties of an algebraic torus, such results are known for complements of hyperplane arrangements \cite{Zharkov2013}, 
    for quasilinear varieties \cite{Schock2025}, 
    and more generally for cohomologically tropical varieties \cite{Aksnes2025}.
    For compact varieties, comparison results are available for smooth compact toric varieties \cite{Amini2024c} and for tropical compactifications of cohomologically tropical varieties \cite{Aksnes2025}.
    
    However, these results do not apply directly to arrangement Schubert varieties in general.
    Indeed, the existing comparison theorems are established in settings corresponding to smooth algebraic varieties, whereas arrangement Schubert varieties are typically singular.
    Moreover, the tropical spaces considered in \cite{Amini2024c,Aksnes2025} are regular at infinity in the sense of \cite{Mikhalkin2014}, whereas tropical matroid Schubert varieties need not be.
\end{remark}

\section{Examples}\label{examples}

In this section, we describe the stratification of $Y_M$ by admissible pairs in several explicit examples.
For simplicity, we write $\{j, \dots , n\}$ as $j \dots n$.
Throughout the examples, we show that $Y_M$ not only contains the information of $L(M)$, 
but also distinguishes loop elements in $M$.
In particular, the construction of $Y_M$ applies to arbitrary matroids, 
not merely to loopless matroids, as is often assumed in the literature.

\subsection{Boolean matroid of rank $2$}
Let $M = U_{2,2}$ be the Boolean matroid of rank $2$ on $E=12$ with the independence complex $\mathcal{I}(M)$ and the lattice of flats $L(M)$ as follows:

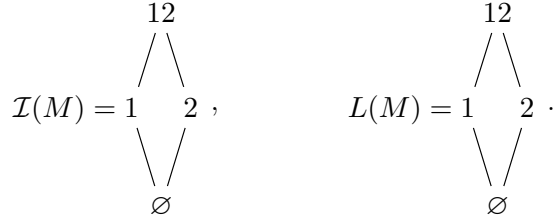
\begin{center}
  
  \begin{tikzcd}[column sep=-0.6em]
    && 12 \\
    {\mathcal{I}(M)=} &  1 && 2 \\
    && \varnothing
    \arrow[no head, from=1-3, to=2-4]
    \arrow[no head, from=2-2, to=1-3]
    \arrow[no head, from=2-2, to=3-3]
    \arrow[no head, from=2-4, to=3-3]
  \end{tikzcd},\qquad\qquad
  \begin{tikzcd}[column sep=-0.6em]
    && 12 \\
    {L(M)=} &  1 && 2 \\
    && \varnothing
    \arrow[no head, from=1-3, to=2-4]
    \arrow[no head, from=2-2, to=1-3]
    \arrow[no head, from=2-2, to=3-3]
    \arrow[no head, from=2-4, to=3-3]
  \end{tikzcd}.
\end{center}

Since $U_M = \mathbb{R}^2$, the variety $Y_M$, which is the closure of $U_M$ in $(\mathbb{TP}^1)^2$, is simply $(\mathbb{TP}^1)^2$.
Alternatively, one can observe that $M$ decomposes as $M=N\oplus O$, where $N$ and $O$ are rank-$1$ Boolean matroids on $1$ and $2$, respectively,
and $Y_N\simeq \mathbb{TP}^1$ and $Y_O\simeq \mathbb{TP}^1$.
By applying \cref{Y_Mproductdecomp}, we again get that $Y_M = (\mathbb{TP}^1)^2$.

The figure below shows the induced polyhedral structure of $Y_M$ by $\Sigma_M^+$, 
together with the stratification by admissible pairs.

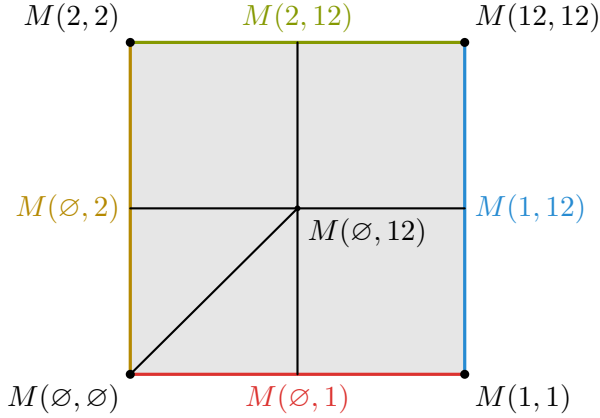
\begin{center}
  \begin{tikzpicture}[scale=2.2, line cap=round, line join=round]

    \def\a{1.0} 
    
    \definecolor{edgeA}{RGB}{220,50,47}   
    \definecolor{edgeB}{RGB}{38,139,210}  
    \definecolor{edgeC}{RGB}{133,153,0}   
    \definecolor{edgeD}{RGB}{181,137,0}   
    
    \tikzset{
      edgebottom/.style={very thick, edgeA},
      edgeright/.style ={very thick, edgeB},
      edgetop/.style   ={very thick, edgeC},
      edgeleft/.style  ={very thick, edgeD},
      edgelabA/.style  ={text=edgeA, font=\normalsize},
      edgelabB/.style  ={text=edgeB, font=\normalsize},
      edgelabC/.style  ={text=edgeC, font=\normalsize},
      edgelabD/.style  ={text=edgeD, font=\normalsize},
      vtxdot/.style    ={circle, fill=black, inner sep=1.2pt},
      vtxlab/.style    ={text=black, font=\normalsize},
      origlab/.style   ={text=black, font=\normalsize}
    }
    
    \coordinate (SW) at (-\a,-\a);
    \coordinate (SE) at ( \a,-\a);
    \coordinate (NE) at ( \a, \a);
    \coordinate (NW) at (-\a, \a);
    
    \fill[black!10] (SW) -- (SE) -- (NE) -- (NW) -- cycle;
    
    \draw[edgebottom] (SW) -- (SE);
    \draw[edgeright]  (SE) -- (NE);
    \draw[edgetop]    (NE) -- (NW);
    \draw[edgeleft]   (NW) -- (SW);
    
    \node[edgelabA, below] at ($(SW)!0.5!(SE)$) {$M(\varnothing,1)$};
    \node[edgelabB, right] at ($(SE)!0.5!(NE)$) {$M(1,12)$};
    \node[edgelabC, above] at ($(NE)!0.5!(NW)$) {$M(2,12)$};
    \node[edgelabD, left]  at ($(NW)!0.5!(SW)$) {$M(\varnothing,2)$};
    
    \node[vtxdot] at (SW) {};
    \node[vtxdot] at (SE) {};
    \node[vtxdot] at (NE) {};
    \node[vtxdot] at (NW) {};
    
    \node[vtxlab, below left]  at (SW) {$M(\varnothing,\varnothing)$};
    \node[vtxlab, below right] at (SE) {$M(1,1)$};
    \node[vtxlab, above right] at (NE) {$M(12,12)$};
    \node[vtxlab, above left]  at (NW) {$M(2,2)$};
    
    \coordinate (O) at (0,0);
    
    \coordinate (E) at (\a,0);
    \coordinate (N) at (0,\a);
    \coordinate (W) at (-\a,0);
    \coordinate (S) at (0,-\a);
    
    \draw[thick] (O) -- (E);
    \draw[thick] (O) -- (N);
    \draw[thick] (O) -- (W);
    \draw[thick] (O) -- (S);
    \draw[thick] (O) -- (SW);
    
    \node[origlab, below right] at (O) {$M(\varnothing,12)$};
    
    \fill (O) circle (0.5pt);
    
    \end{tikzpicture}
\end{center}

\subsection{A matroid with parallel elements}\label{ex2}
Let $M$ be a matroid on $E=123$ with the independence complex $\mathcal{I}(M)$ and the lattice of flats $L(M)$ as follows:

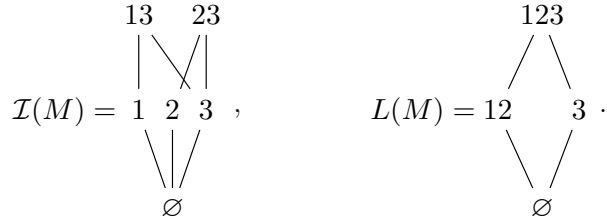
\begin{center}
  
\begin{tikzcd}[column sep=-0.6em]
	& 13 && 23 \\
	{\mathcal{I}(M)=} & 1 & 2 & 3 \\
	&& \varnothing
	\arrow[no head, from=1-2, to=2-4]
	\arrow[no head, from=1-4, to=2-4]
	\arrow[no head, from=2-2, to=1-2]
	\arrow[no head, from=2-2, to=3-3]
	\arrow[no head, from=2-3, to=1-4]
	\arrow[no head, from=2-3, to=3-3]
	\arrow[no head, from=2-4, to=3-3]
\end{tikzcd},\qquad\qquad
\begin{tikzcd}[column sep=-0.6em]
	&& 123 \\
	{L(M)=} &  12 && 3 \\
	&& \varnothing
	\arrow[no head, from=1-3, to=2-4]
	\arrow[no head, from=2-2, to=1-3]
	\arrow[no head, from=2-2, to=3-3]
	\arrow[no head, from=2-4, to=3-3]
\end{tikzcd}.
\end{center}
Note that $1$ and $2$ are parallel elements, i.e., $cl(1) = cl(2) = 12$; and $3$
is a coloop, i.e., $E \setminus 3 = 12 \in L(M)$. 
Hence we may decompose $M$ into $M=N\oplus O$,
where 

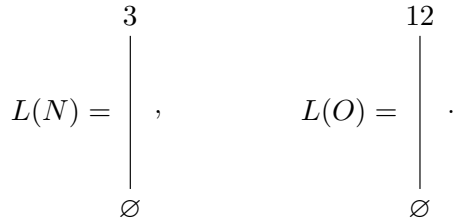
\begin{center}
  \begin{tikzcd}[column sep=-0.6em]
	& 3 \\
	{L(N)=} \\
	& \varnothing
	\arrow[no head, from=1-2, to=3-2]
\end{tikzcd}, \qquad \qquad
\begin{tikzcd}[column sep=-0.6em]
	& 12 \\
	{L(O)=} \\
	& \varnothing
	\arrow[no head, from=1-2, to=3-2]
\end{tikzcd}.
\end{center}
Here $N$ is the rank-$1$ Boolean matroid on $3$, and $O$ is the matroid on $12$ consisting only parallel elements.

The figure below is the stratification of $Y_M$ by admissible pairs.

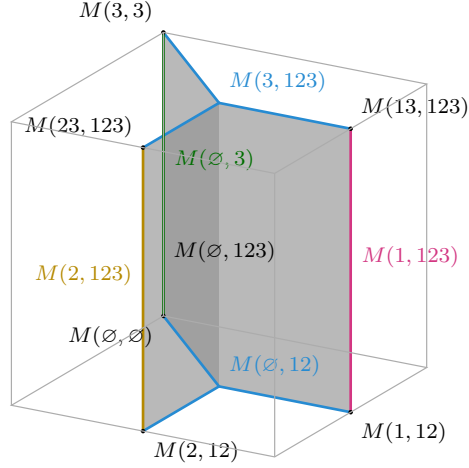
\begin{center}

\tdplotsetmaincoords{70}{120}

\newcommand{\MRAYS}{M(\varnothing,12)}
\newcommand{\UM}{M(\varnothing,123)}

\newcommand{\MAzero}{M(\varnothing,\varnothing)}
\newcommand{\MBzero}{M(2,12)}
\newcommand{\MCzero}{M(1,12)}

\newcommand{\MAone}{M(3,3)}
\newcommand{\MBone}{M(23,123)} 
\newcommand{\MCone}{M(13,123)}

\newcommand{\MeThreeA}{M(\varnothing,3)}
\newcommand{\MeThreeB}{M(2,123)}
\newcommand{\MeThreeC}{M(1,123)}

\begin{tikzpicture}[scale=4, tdplot_main_coords, line cap=round, line join=round]

\pgfdeclarelayer{bg}
\pgfdeclarelayer{fg}
\pgfsetlayers{bg,main,fg}

\definecolor{raycol}{RGB}{38,139,210}
\definecolor{cE3A}{RGB}{0,110,0}
\definecolor{cE3B}{RGB}{181,137,0}
\definecolor{cE3C}{RGB}{211,54,130}

\tikzset{
  cubeedge/.style={draw=gray!65, line width=0.45pt},
  ray/.style={line width=1.0pt, draw=raycol},
  raylabel/.style={text=raycol, font=\scriptsize},
  vsegA/.style={line width=1.0pt, draw=cE3A},
  vsegB/.style={line width=1.0pt, draw=cE3B},
  vsegC/.style={line width=1.0pt, draw=cE3C},
  vtxdot/.style={circle, fill=black, inner sep=0.6pt},
  vtxlab/.style={text=black, font=\scriptsize},
  vtxlabRay/.style={font=\scriptsize}, 
  midlab/.style={text=black, font=\scriptsize}
}

\coordinate (000) at (0,0,0);
\coordinate (100) at (1,0,0);
\coordinate (010) at (0,1,0);
\coordinate (110) at (1,1,0);

\coordinate (001) at (0,0,1);
\coordinate (101) at (1,0,1);
\coordinate (011) at (0,1,1);
\coordinate (111) at (1,1,1);

\coordinate (O0) at (1/2,1/2,0);
\coordinate (O1) at (1/2,1/2,1);

\coordinate (A0) at (0,0,0);
\coordinate (B0) at (1,1/2,0);
\coordinate (C0) at (1/2,1,0);

\coordinate (A1) at (0,0,1);
\coordinate (B1) at (1,1/2,1);
\coordinate (C1) at (1/2,1,1);

\begin{pgfonlayer}{bg}
  \fill[gray!85, opacity=0.65] (O0)--(A0)--(A1)--(O1)--cycle;
  \fill[gray!85, opacity=0.65] (O0)--(B0)--(B1)--(O1)--cycle;
  \fill[gray!85, opacity=0.65] (O0)--(C0)--(C1)--(O1)--cycle;
\end{pgfonlayer}

\draw[ray] (O0)--(A0);
\draw[ray] (O0)--(B0);
\draw[ray] (O0)--(C0);

\draw[ray] (O1)--(A1);
\draw[ray] (O1)--(B1);
\draw[ray] (O1)--(C1);

\node[raylabel, above right] at (O0) {$\MRAYS$};
\node[raylabel, above right] at (O1) {$M(3,123)$};
\node[midlab] at (0.55,0.55,0.5) {$\UM$};

\draw[vsegA] (A0)--(A1) node[pos=0.55, right,  text=cE3A, font=\scriptsize] {$\MeThreeA$};
\draw[vsegB] (B0)--(B1) node[pos=0.55, left, text=cE3B, font=\scriptsize] {$\MeThreeB$};
\draw[vsegC] (C0)--(C1) node[pos=0.55, right, text=cE3C, font=\scriptsize] {$\MeThreeC$};

\node[vtxdot] at (A0) {}; \node[vtxlab, below left]  at (A0) {$\MAzero$};
\node[vtxdot] at (B0) {}; \node[vtxlab, below right] at (B0) {$\MBzero$};
\node[vtxdot] at (C0) {}; \node[vtxlab, below right] at (C0) {$\MCzero$};

\node[vtxdot] at (A1) {}; \node[vtxlab,    above left]  at (A1) {$\MAone$};
\node[vtxdot] at (B1) {}; \node[vtxlabRay, above left] at (B1) {$\MBone$};
\node[vtxdot] at (C1) {}; \node[vtxlab,    above right] at (C1) {$\MCone$};

\begin{pgfonlayer}{fg}
  \draw[cubeedge] (000)--(100)--(110)--(010)--cycle;
  \draw[cubeedge] (001)--(101)--(111)--(011)--cycle;
  \draw[cubeedge] (000)--(001);
  \draw[cubeedge] (100)--(101);
  \draw[cubeedge] (010)--(011);
  \draw[cubeedge] (110)--(111);
\end{pgfonlayer}

\end{tikzpicture}
\end{center}
One can directly see that $Y_M = Y_N \times Y_O$ by observing that $M(23,123)=N$ and $M(\varnothing,12)=O$.
Also note that $M(1,123) \simeq M(2,123)$ but they appear separately in the stratification.
This shows that it is important to distinguish admissible pairs even 
when the corresponding minors are isomorphic. 

Moreover, $M(1,123)$ is a matroid on $23$ in which $2$ is a loop; 
geometrically, this is reflected by the fact that the corresponding 
stratum has no $e_2$-coordinate. 
Similarly, $M(2,123)$ is a matroid on $13$ in which $1$ is a loop, 
and the corresponding stratum has no $e_1$-coordinate. 
Here $e_i$ denotes the $i$-th standard basis vector of $\mathbb{R}^E$.

\appendix
\section{Induced polyhedral structure of a compactification}\label{toriccomp}

In this section, we discuss the induced polyhedral structure of a compactification.
This section is used in \cref{geomofYM} to define the induced polyhedral structure of $Y_M$ inside $(\mathbb{TP}^1)^E$.
Let $U$ be a fanfold in $\mathbb{R}^n$, $\Delta$ be a $d$-dimensional simplicial fan in $\mathbb{R}^n$ such that 
$U \subseteq |\Delta|$. 
Let $Y$ be the compactification of $U$ by $\Delta$.

When $U=|\Delta|$, $Y$ is the canonical compactification $\overline{\Delta}$ of $\Delta$.
In this case, its stratification can be induced from the 
torus orbit stratification of $\mathbb{TP}_\Delta = \coprod_{\eta \in \Delta} O(\eta)$ by
\[\overline{\Delta} = \coprod_{\eta \in \Delta} (O(\eta)\cap \overline{\Delta}) \] 
which can be understood easily by the fan structure of $\Delta$. 
Moreover, a polyhedral structure of $\overline{\Delta}$ can be induced from 
the fan structure of $\Delta$. 
More precisely, a polyhedral structure of $\overline{\Delta}$ can be given as follows. 
Let $\zeta \preceq \eta \in \Delta$. For the projection $\pi^\zeta : \mathbb{R}^n \to O(\zeta)$
defined in \cref{tropvar},
let $\square^{\zeta}_{\eta}$ be the closure of $\pi^\zeta(\eta)$ in $O(\zeta)$. 
Then $\{\square^\zeta_{\eta} \,\vert\, \zeta \preceq \eta\}$ is a polyhedral structure of $\overline{\Delta}$. 
This is the usual induced polyhedral structure on the canonical compactification.
See \cite[Section 2.3]{Amini2024c} for more details.

For a general compactification $Y$,
$O(\eta) \cap Y$ may be empty for $\eta \in \Delta$.
Let $U\subsetneq|\Delta|$. 
Our goal is to give a polyhedral structure on $Y$ induced from a fan structure of $U$. 
We first consider the case when the tropical toric variety is affine.
Fix $\eta \in \Delta$ and let $U_\eta \coloneq U \cap \eta$.
Let $A_\eta$ be the affine tropical toric variety by $\eta$. Recall that for a simplicial cone $\eta$, 
$A_\eta \simeq \mathbb{T}^{|\eta|} \times \mathbb{R}^{n-|\eta|}$.
The affine tropical toric variety $A_\eta$ admits a stratification $A_\eta=\coprod_{\zeta \preceq \eta}O(\zeta)$
    and the closure of $O(\zeta)$ in $A_\eta$ denoted by $V(\zeta)$ where is 
    $V(\zeta) = \coprod_{\zeta\preceq \upsilon \preceq \eta} O(\upsilon)$.
Let $\sigma \subseteq  U_\eta$ be a cone. Let $\overline{\sigma}$ be the closure of $\sigma$ in $A_\eta$
and $\overline{\sigma}^\zeta \coloneq\overline{\sigma} \cap O(\zeta)$ for $\zeta \preceq \eta$.

\begin{proposition}\label{projclos}
Let the setup be as above.
The set $\overline{\sigma}^\zeta$ is non-empty if and only if $\operatorname{relint}(\zeta) \cap \sigma \neq \varnothing$
and moreover, when $\overline{\sigma}^\zeta$ is non-empty, $\overline{\sigma}^\zeta = \pi^\zeta(\sigma)$.
\end{proposition}

\begin{proof}
  Assume first that $\operatorname{relint}(\zeta)\cap \sigma\neq \varnothing$.
  We prove $\overline{\sigma}^\zeta \subseteq \pi^\zeta(\sigma)$.
  For simplicity, set $\pi\coloneq \pi^\zeta$.
  Since $\sigma \subseteq \pi^{-1}(\pi(\sigma)) \subseteq O(\underline{0})$, we have
  \[
  \overline{\sigma}^\zeta=\overline{\sigma}\cap O(\zeta)
  \subseteq \overline{\pi^{-1}(\pi(\sigma))}\cap O(\zeta),
  \]
  where the closure is taken in $A_\eta$.
  Using the identification $O(\underline{0})\simeq O(\zeta)\times \langle\zeta\rangle$ for which $\pi$ is the projection
  onto the first factor, we have $\pi^{-1}(\pi(\sigma))=\pi(\sigma)\times \langle\zeta\rangle$.
  We claim that
  \[
  \overline{\pi(\sigma)\times \langle\zeta\rangle}\cap O(\zeta)=\pi(\sigma).
  \]
  Let $p\in\pi(\sigma)$. Choose a sequence $z_i\in \operatorname{relint}(\zeta)\subset \langle\zeta\rangle$ with $\|z_i\|\to\infty$.
  Then $(p,z_i)\in \pi(\sigma)\times \langle\zeta\rangle$ and $(p,z_i)\to p$ in $A_\eta$, hence
  $p\in \overline{\pi(\sigma)\times \langle\zeta\rangle}\cap O(\zeta)$.
  Conversely, consider $O(\zeta)\setminus \pi(\sigma)$.
  If $O(\zeta) =  \pi(\sigma)$, then we are done. Now let $p \in O(\zeta)\setminus \pi(\sigma)$.
  Since $\sigma$ is a polyhedral cone and $\pi$ is linear, $\pi(\sigma)$ is a polyhedral cone in $O(\zeta)$, hence closed.
  Thus there exists an open neighbourhood $W\subseteq O(\zeta)$ of $p$ such that $W\cap \pi(\sigma)=\varnothing$.
  Then $W\times \overline{\langle\zeta\rangle}$ is an open neighbourhood of $p$ in $A_\eta$ disjoint from
  $\pi(\sigma)\times \langle\zeta\rangle$, so $p\notin \overline{\pi(\sigma)\times \langle\zeta\rangle}$.
  This proves the claim, and therefore $\overline{\sigma}^\zeta\subseteq \pi(\sigma)$.
  
  We now prove $\overline{\sigma}^\zeta \supseteq \pi(\sigma)$ under the same assumption $\operatorname{relint}(\zeta)\cap \sigma\neq \varnothing$.
  Choose $v\in \operatorname{relint}(\zeta)\cap \sigma$.
  For any $q\in \sigma$ and $t\ge 0$, we have $q+t v\in \sigma$ since $\sigma$ is a cone.
  Moreover $\pi(q+t v)=\pi(q)$ and $t v$ goes to infinity in the $\zeta$-directions, hence $q+t v\to \pi(q)$ in $A_\eta$.
  Therefore $\pi(q)\in \overline{\sigma}\cap O(\zeta)=\overline{\sigma}^\zeta$ for all $q\in\sigma$, and thus $\pi(\sigma)\subseteq \overline{\sigma}^\zeta$.
  Combining with the previous inclusion, we obtain $\overline{\sigma}^\zeta=\pi(\sigma)$, in particular $\overline{\sigma}^\zeta\neq \varnothing$.
  
  Finally, assume $\overline{\sigma}^\zeta\neq \varnothing$.
  Since $\overline{\sigma}^\zeta$ is a closure of a cone, it contains $\infty_\zeta$. 
  Suppose $\operatorname{relint}(\zeta)\cap \sigma =\varnothing$. Let $(x_i)$ be a sequence in $\sigma$ that converges to $\infty_\zeta$.
  We may decompose $x_i$ by $x_i = u_i\times z_i \in O(\zeta)\times \langle \zeta \rangle $.
  Since $\operatorname{relint}(\zeta)\cap \sigma =\varnothing$ and $\sigma$ is closed, $u_i$ cannot converge to zero which gives a contradiction.
  This proves that $\overline{\sigma}^\zeta$ is non-empty if and only if $\operatorname{relint}(\zeta)\cap \sigma\neq \varnothing$,
  and in the non-empty case we have $\overline{\sigma}^\zeta=\pi^\zeta(\sigma)$.
\end{proof}

\begin{corollary}\label{clsigma}
    Let $\sigma \subseteq U_\eta$ be as above. Then 
    \[\overline{\sigma} = \coprod_{\substack{\zeta\preceq \eta\\\operatorname{relint}(\zeta)\cap \sigma \neq \varnothing}}\overline{\sigma}^\zeta.\]
\end{corollary}

\begin{proof}
  This follows directly from the stratification of $A_\eta$.
\end{proof} 

\begin{definition}[\cref{compfan}]
  Let $\Delta$ be a simplicial fan in $\mathbb{R}^n$, and let $\Sigma$ be a fan 
with the support $U\subseteq |\Delta|$. For a cone $\eta\in\Delta$, define the restriction
\[
\Sigma|_\eta \coloneq \{\sigma\in \Sigma \,\vert\, \sigma \subseteq \eta\}.
\]
We say $\Sigma$ is \emph{$\Delta$-compatible} if 
$|\Sigma|_\eta| = U\cap\eta$ for all $\eta\in\Delta$.

\end{definition}
  
  Given a $\Delta$-compatible fan $\Sigma$ and $\eta\in\Delta$, we denote by $\Sigma_\eta$
  a subfan of $\Sigma$ with support $U_\eta$.

\begin{remark}
    A $\Delta$-compatible fan always exists since we can take a common refinement of any fan $\Sigma$ with support $U$ and $\Delta$.
\end{remark}
    
Noting that $X = \mathbb{TP}_\Delta$ is defined by gluing $A_\eta$ for $\eta \in \Delta$,
we may understand $Y$, the compactification of $U$ by $\Delta$, 
by gluing $\overline{\sigma}$ accordingly.

\begin{definition}[\cref{indpolcom}]
  Let $\Sigma$ be a $\Delta$-compatible fan. The \emph{induced polyhedral structure} $\mathcal{Y}$ of $Y$ by $\Sigma$ is the collection of faces
  \[
  \mathcal{Y}=\{\pi^\zeta(\sigma)\mid \operatorname{relint}(\zeta)\cap \sigma\neq\varnothing,\ \sigma\in\Sigma,\ \zeta\in\Delta\}.
  \]
\end{definition}

\begin{proposition}
 The induced polyhedral structure $\mathcal{Y}$ by a $\Delta$-compatible fan $\Sigma$ is a polyhedral complex.
\end{proposition}

\begin{proof}
  Let $U = |\Sigma|$.
  The compatibility condition ensures that for any cone $\sigma \in \Sigma$, we can find a cone $\eta \in \Delta$ such that $\sigma \subseteq  U_\eta$ holds.
  Hence we can apply \cref{projclos} and get a polyhedral structure on each affine chart $A_\eta$.
  This induces a polyhedral structure on $Y$ by gluing $\pi^\zeta(\sigma)$ compatibly along overlaps.
\end{proof}

\begin{remark}
  The presentation $\pi^\zeta(\sigma)$ need not be unique, that is, there may exist $\sigma,\sigma'\in\Sigma$ with $\sigma\neq\sigma'$ 
  such that
  $\pi^\zeta(\sigma)=\pi^\zeta(\sigma')$.
\end{remark}

\begin{corollary}
  Let the setup be as above. Then
  \[
  Y=\coprod_{\substack{\eta\in\Delta\\ Y\cap O(\eta)\neq\varnothing}} (Y\cap O(\eta))
  =\coprod_{\substack{\eta\in\Delta\\ \operatorname{relint}(\eta)\cap U\neq\varnothing}} (Y\cap O(\eta)).
  \]
\end{corollary}

\begin{proof}
  Fix a $\Delta$-compatible fan $\Sigma$. By \cref{projclos}, we have $Y\cap O(\eta)\neq\varnothing$ if and only if
  $\operatorname{relint}(\eta)\cap U\neq\varnothing$. This gives the stated decomposition.
\end{proof}

Let $Y^\eta \coloneq Y \cap O(\eta)$. 
Then $Y^\eta$ is a fanfold in $O(\eta)$ since we can induce a fan structure from $\Sigma$.
We call \[ Y=\coprod_{\substack{\eta\in\Delta\\ \operatorname{relint}(\eta)\cap U\neq\varnothing}} Y^\eta\]
the \emph{sedentary decomposition} of $Y$ by $\Delta$.

\bibliographystyle{alpha}
\bibliography{ref}

@Article{Zharkov2013,
  author  = {Zharkov, Ilia},
  journal = {Journal of G{\"o}kova Geometry Topology},
  title   = {The {O}rlik-{S}olomon algebra and the {B}ergman fan of a matroid},
  year    = {2013},
  pages   = {25--31},
  volume  = {7},
  ids     = {Zharkov},
}

@InCollection{Shaw2013a,
  author    = {Shaw, Kristin},
  booktitle = {Algebraic and Combinatorial Aspects of Tropical Geometry},
  publisher = {American Mathematical Society},
  title     = {Tropical (1,1)-homology for floor decomposed surfaces},
  year      = {2013},
  address   = {Providence, {RI}},
  editor    = {{E. Brugallé, M. A. Cueto, A. Dicken-stein, E.M. Feichtner, and I. Itenberg, editors}},
  pages     = {529--550},
  series    = {Contemporary mathematics},
  volume    = {589},
}

@Article{Payne2009,
  author    = {Payne, Sam},
  journal   = {Mathematical Research Letters},
  title     = {Analytification is the limit of all tropicalizations},
  year      = {2009},
  number    = {3},
  pages     = {543--556},
  volume    = {16},
  ids       = {Payne},
  publisher = {International Press of Boston},
}

@Book{Oxley2011,
  author     = {Oxley, James},
  publisher  = {Oxford University Press, New York},
  title      = {Matroid Theory},
  year       = {2011},
  collection = {Oxford graduate texts in mathematics},
}

@InCollection{Mikhalkin2014,
  author    = {Mikhalkin, Grigory and Zharkov, Ilia},
  booktitle = {Homological mirror symmetry and tropical geometry},
  publisher = {Springer},
  title     = {Tropical eigenwave and intermediate {J}acobians},
  year      = {2014},
  pages     = {309--349},
  series    = {Lecture Notes of the Unione Matematica Italiana},
  ids       = {MZ},
}

@Book{Maclagan2015,
  author    = {Maclagan, Diane and Sturmfels, Bernd},
  publisher = {American Mathematical Soc.},
  title     = {Introduction to tropical geometry},
  year      = {2015},
  volume    = {161},
  ids       = {MS, MaclaganSturmfels},
}

@Article{Kajiwara2008,
  author    = {Kajiwara, Takeshi},
  journal   = {Contemporary Mathematics},
  title     = {Tropical toric geometry},
  year      = {2008},
  pages     = {197--208},
  volume    = {460},
  ids       = {Kaj},
  publisher = {Providence, RI: American Mathematical Society},
}

@Article{Jell2018,
  author  = {Jell, Philipp and Rau, Johannes and Shaw, Kristin},
  journal = {{\'E}pijournal de G{\'e}om{\'e}trie Alg{\'e}brique},
  title   = {Lefschetz $(1, 1)$-theorem in tropical geometry},
  year    = {2018},
  pages   = {Art. 11, 27},
  volume  = {2},
  ids     = {Lefschetz11},
}

@Article{Jell2019,
  author    = {Jell, Philipp and Shaw, Kristin and Smacka, Jascha},
  journal   = {Advances in Geometry},
  title     = {Superforms, tropical cohomology, and {P}oincar{\'e} duality},
  year      = {2019},
  number    = {1},
  pages     = {101--130},
  volume    = {19},
  publisher = {De Gruyter},
}

@Article{Itenberg2019,
  author    = {Itenberg, Ilia and Katzarkov, Ludmil and Mikhalkin, Grigory and Zharkov, Ilia},
  journal   = {Mathematische Annalen},
  title     = {Tropical homology},
  year      = {2019},
  number    = {1-2},
  pages     = {963--1006},
  volume    = {374},
  publisher = {Springer},
}

@Book{Fulton1984,
  author    = {Fulton, William},
  publisher = {Springer-Verlag, Berlin},
  title     = {Intersection theory},
  year      = {1984},
  series    = {Ergebnisse der Mathematik und ihrer Grenzgebiete (3) [Results in Mathematics and Related Areas (3)]},
  volume    = {2},
  ids       = {FultonIntersectionTheory},
  pages     = {xi+470},
}

@Article{Braden2020,
  author    = {Braden, Tom and Huh, June and Matherne, Jacob and Proudfoot, Nicholas and Wang, Botong},
  journal   = {arXiv:2010.06088},
  title     = {Singular {Hodge} theory for combinatorial geometries},
  year      = {2020},
  _language = {en},
  _pages    = {95},
  abstract  = {We introduce the intersection cohomology module of a matroid and prove that it satisﬁes Poincare´ duality, the hard Lefschetz theorem, and the Hodge–Riemann relations. As applications, we obtain proofs of Dowling and Wilson’s Top-Heavy conjecture and the nonnegativity of the coefﬁcients of Kazhdan–Lusztig polynomials for all matroids.},
}

@Article{Amini2023,
  author  = {Amini, Omid and Piquerez, Matthieu},
  journal = {arXiv:2310.15367 to appear in Bulletin de la Société Mathématique de France},
  title   = {Hodge theory for tropical fans},
  year    = {2023},
}

@Article{Amini2020,
  author  = {Amini, Omid and Piquerez, Matthieu},
  journal = {arXiv:2007.07826},
  title   = {Hodge theory for tropical varieties},
  year    = {2020},
  ids     = {AP-hodge, AminiPiquerezHodge},
}

@Article{Adiprasito2018,
  author    = {Adiprasito, Karim and Huh, June and Katz, Eric},
  journal   = {Annals of Mathematics},
  title     = {Hodge theory for combinatorial geometries},
  year      = {2018},
  number    = {2},
  pages     = {381--452},
  volume    = {188},
  publisher = {JSTOR},
}

@Article{MR4477425,
  author     = {Braden, Tom and Huh, June and Matherne, Jacob P. and Proudfoot, Nicholas and Wang, Botong},
  journal    = {Adv. Math.},
  title      = {A semi-small decomposition of the {C}how ring of a matroid},
  year       = {2022},
  issn       = {0001-8708,1090-2082},
  pages      = {Paper No. 108646, 49},
  volume     = {409},
  doi        = {10.1016/j.aim.2022.108646},
  fjournal   = {Advances in Mathematics},
  mrclass    = {05B35 (14C25)},
  mrnumber   = {4477425},
  mrreviewer = {Joseph\ Kung},
}

@Article{Amini2024b,
  author        = {Amini, Omid and Piquerez, Matthieu},
  journal       = {arXiv:2405.05718},
  title         = {Homological smoothness and {D}eligne resolution for tropical fans},
  year          = {2024},
  archiveprefix = {arXiv},
  copyright     = {Creative Commons Attribution Non Commercial Share Alike 4.0 International},
  doi           = {10.48550/ARXIV.2405.05718},
  eprint        = {2405.05718},
  primaryclass  = {math.AG},
  publisher     = {arXiv},
}

@Article{Amini2024c,
  author        = {Amini, Omid and Piquerez, Matthieu},
  journal       = {arXiv:2405.05014},
  title         = {Tropical {F}eichtner-{Y}uzvinsky and positivity criterion for fans},
  year          = {2024},
  archiveprefix = {arXiv},
  copyright     = {Creative Commons Attribution Non Commercial Share Alike 4.0 International},
  doi           = {10.48550/ARXIV.2405.05014},
  eprint        = {2405.05014},
  primaryclass  = {math.AG},
  publisher     = {arXiv},
}

@Article{Huh2017,
  author   = {Huh, June and Wang, Botong},
  journal  = {Acta Math.},
  title    = {Enumeration of points, lines, planes, etc.},
  year     = {2017},
  number   = {2},
  pages    = {297-317},
  volume   = {218},
  doi      = {10.4310/ACTA.2017.v218.n2.a2},
  mrnumber = {3733101},
}

@Article{Dowling1974,
  author   = {Dowling, Thomas A. and Wilson, Richard M.},
  journal  = {Trans. Amer. Math. Soc.},
  title    = {The slimmest geometric lattices},
  year     = {1974},
  pages    = {203-215},
  volume   = {196},
  doi      = {10.2307/1997023},
  mrnumber = {345849},
}

@Article{Dowling1975,
  author   = {Dowling, Thomas A. and Wilson, Richard M.},
  journal  = {Proc. Amer. Math. Soc.},
  title    = {Whitney number inequalities for geometric lattices},
  year     = {1975},
  pages    = {504-512},
  volume   = {47},
  doi      = {10.2307/2039773},
  mrnumber = {354422},
}

@Article{Eur2023,
  author   = {Eur, Christopher and Huh, June and Larson, Matt},
  journal  = {Forum Math. Pi},
  title    = {Stellahedral geometry of matroids},
  year     = {2023},
  pages    = {Paper No. e24, 48 pp.},
  volume   = {11},
  doi      = {10.1017/fmp.2023.24},
  mrnumber = {4653766},
}

@Article{Aksnes2025,
  author   = {Aksnes, Edvard and Amini, Omid and Piquerez, Matthieu and Shaw, Kris},
  journal  = {J. Inst. Math. Jussieu},
  title    = {COHOMOLOGICALLY TROPICAL VARIETIES},
  year     = {2025},
  number   = {6},
  pages    = {2543-2572},
  volume   = {24},
  doi      = {10.1017/S1474748025101114},
  mrnumber = {4976580},
}

@Article{Brylawski,
  author   = {Brylawski, Thomas},
  journal  = {Matroid theory and its applications. Liguori editore, Naples.},
  title    = {The {T}utte polynomial. {I}. {G}eneral theory},
  year     = {1982},
  pages    = {125-275},
  mrnumber = {863010},
}

@Article{Shaw2023,
  author   = {Shaw, Kris and Werner, Annette},
  journal  = {Comb. Theory},
  title    = {On the birational geometry of matroids},
  year     = {2023},
  number   = {2},
  pages    = {Paper No. 17, 32 pp.},
  volume   = {3},
  doi      = {10.5070/c63261996},
  mrnumber = {4646098},
}

@Article{Gross2023,
  author   = {Gross, Andreas and Shokrieh, Farbod},
  journal  = {J. Algebra},
  title    = {A sheaf-theoretic approach to tropical homology},
  year     = {2023},
  pages    = {577-641},
  volume   = {635},
  doi      = {10.1016/j.jalgebra.2023.08.014},
  mrnumber = {4637248},
}

@Article{Ardila2016,
  author   = {Ardila, Federico and Boocher, Adam},
  journal  = {J. Algebraic Combin.},
  title    = {The closure of a linear space in a product of lines},
  year     = {2016},
  number   = {1},
  pages    = {199-235},
  volume   = {43},
  doi      = {10.1007/s10801-015-0634-x},
  mrnumber = {3439307},
}

@Article{Braden2025,
  author        = {Braden, Tom and Proudfoot, Nicholas},
  journal       = {arXiv:2510.09488 to appear in the Proceedings of the ICM 2026},
  title         = {Intersection cohomology without spaces},
  year          = {2025},
  archiveprefix = {arXiv},
  copyright     = {Creative Commons Attribution 4.0 International},
  doi           = {10.48550/ARXIV.2510.09488},
  eprint        = {2510.09488},
  primaryclass  = {math.AG},
  publisher     = {arXiv},
}

@Article{Bj_rner2009,
  author   = {Björner, Anders and Ekedahl, Torsten},
  journal  = {Ann. of Math. (2)},
  title    = {On the shape of {B}ruhat intervals},
  year     = {2009},
  number   = {2},
  pages    = {799-817},
  volume   = {170},
  doi      = {10.4007/annals.2009.170.799},
  mrnumber = {2552108},
}

@Article{beilinson1982faisceaux,
  author  = {Beilinson, Alexander and Bernstein, Joseph and Deligne, Pierre and Gabber, Ofer},
  journal = {Ast{\'e}risque},
  title   = {Faisceaux pervers},
  year    = {1982},
  pages   = {7--171},
  volume  = {100},
}

@Article{Fiebig2014,
  author   = {Fiebig, Peter and Williamson, Geordie},
  journal  = {Ann. Inst. Fourier (Grenoble)},
  title    = {Parity sheaves, moment graphs and the $p$-smooth locus of {S}chubert varieties},
  year     = {2014},
  number   = {2},
  pages    = {489-536},
  volume   = {64},
  doi      = {10.5802/aif.2856},
  mrnumber = {3330913},
}

@Article{Ardila2006,
  author   = {Ardila, Federico and Klivans, Caroline J.},
  journal  = {J. Combin. Theory Ser. B},
  title    = {The {B}ergman complex of a matroid and phylogenetic trees},
  year     = {2006},
  number   = {1},
  pages    = {38-49},
  volume   = {96},
  doi      = {10.1016/j.jctb.2005.06.004},
  mrnumber = {2185977},
}

@Article{Schock2025,
  author   = {Nolan Schock},
  journal  = {Advances in Mathematics},
  title    = {Quasilinear tropical compactifications},
  year     = {2025},
  issn     = {0001-8708},
  pages    = {110037},
  volume   = {461},
  doi      = {https://doi.org/10.1016/j.aim.2024.110037},
  url      = {https://www.sciencedirect.com/science/article/pii/S000187082400553X},
}

@Article{MR1503085,
  author   = {Whitney, Hassler},
  journal  = {Ann. of Math. (2)},
  title    = {The coloring of graphs},
  year     = {1932},
  issn     = {0003-486X,1939-8980},
  number   = {4},
  pages    = {688--718},
  volume   = {33},
  doi      = {10.2307/1968214},
  fjournal = {Annals of Mathematics. Second Series},
  mrclass  = {99-04},
  mrnumber = {1503085},
  url      = {https://doi.org/10.2307/1968214},
}
\end{document}